\documentclass[a4paper,11pt]{amsart}

\protected\def\ignorethis#1\endignorethis{}
\let\endignorethis\relax

\def\TOCstop{\addtocontents{toc}{\ignorethis}}
\def\TOCstart{\addtocontents{toc}{\endignorethis}}

\usepackage[backend=bibtex, 
style=alphabetic,
sortcites=true,
firstinits=true,
hyperref,
maxbibnames=99,
]{biblatex}

\bibliography{Bibliography}{}

\usepackage[top=3.0cm, bottom=3.0cm, inner=3.0cm, outer=3.0cm, includefoot]{geometry}

\usepackage[latin1]{inputenc}
\usepackage[T1]{fontenc}
\usepackage{verbatim}

\usepackage{geometry}
\usepackage{amssymb}
\usepackage{amsmath}
\usepackage{graphicx}
\usepackage{amsthm}

\usepackage{color}

\usepackage{enumerate}

\usepackage[mathscr]{euscript}

\usepackage{tikz}
\usepackage{tkz-graph}
\tikzstyle{vertex}=[circle, draw, inner sep=0pt, minimum size=6pt]
\newcommand{\vertex}{\node[vertex]}
\usetikzlibrary{fit}

\newcommand{\IR}{{\mathbb{R}}}

\newcommand{\IN}{{\mathbb{N}}}
\newcommand{\IZ}{{\mathbb{Z}}}

\newcommand{\eChar}{\begin{enumerate}[(i)]}
\newcommand{\eCharR}{\begin{enumerate}[(a)]}
\newcommand{\eBr}{\begin{enumerate}[(1)]}

\newcommand{\Deg}{\operatorname{Deg}}

\newcommand{\diam}{\operatorname{diam}}

\newcommand{\supp}{\operatorname{supp}}
\newcommand{\eps}{\varepsilon}

\newcommand{\Abstract}

\newcommand{\Hm}[1]{\leavevmode{\marginpar{\tiny%
$\hbox to 0mm{\hspace*{-0.5mm}$\leftarrow$\hss}%
\vcenter{\vrule depth 0.1mm height 0.1mm width \the\marginparwidth}%
\hbox to 0mm{\hss$\rightarrow$\hspace*{-0.5mm}}$\\\relax\raggedright
#1}}}
\newcommand{\eat}[1]{}

\title
{
Ollivier Ricci curvature for general graph Laplacians: Heat equation, Laplacian comparison, non-explosion and diameter bounds 
}

\author{Florentin M\"unch, Rados{\l}aw K. Wojciechowski}

\usepackage{authblk}

\date{\today}

\theoremstyle{plain}
\newtheorem{lemma}{Lemma}[section]
\newtheorem{theorem}[lemma]{Theorem}
\newtheorem{proposition}[lemma]{Proposition}
\newtheorem{corollary}[lemma]{Corollary}

\theoremstyle{definition}

\newtheorem{example}[lemma]{Example}

\newtheorem{rem}[lemma]{Remark}
\newtheorem{defn}[lemma]{Definition}

\numberwithin{equation}{section}

\begin{document}

\maketitle





\pagestyle{plain}

\begin{abstract}
Discrete time random walks on a finite set naturally translate via a one-to-one correspondence to discrete Laplace operators.
Typically, Ollivier curvature has been investigated via random walks. 
We first extend the definition of Ollivier curvature to general weighted graphs and then 
give a strikingly simple representation of Ollivier curvature using the graph Laplacian.
Using the Laplacian as a generator of a continuous time Markov chain, we connect Ollivier curvature with the heat equation which is strongly related to continuous time random walks. In particular, we prove that a lower bound on the Ollivier curvature is equivalent to a certain Lipschitz decay of solutions to the heat equation. 
This is a discrete analogue to a celebrated Ricci curvature lower bound characterization by Renesse and Sturm.
Our representation of Ollivier curvature via the Laplacian allows us to deduce a Laplacian comparison principle by which we prove non-explosion and improved diameter bounds.
\end{abstract}

\tableofcontents

\section{Introduction}

Ricci curvature is strongly related to the heat equation.
In particular, lower Ricci curvature bounds can be characterized via gradient estimates as in the following theorem by Renesse and Sturm
(see \cite[Theorem~1.3 and Corollary~1.4]{renesse2005transport}). 
	
 \begin{theorem}[Ricci curvature and gradient estimates]\label{thm:Sturm}
For any smooth connected Riemannian manifold $M$ and any $K \in \IR$ the following properties are equivalent:
\begin{enumerate}
\item[(i)]
 $Ric(M) \geq K$.
 \item[(ii)] 
  For all $f \in C_c^{\infty}(M)$ and all $t > 0$
 $$
 \|\nabla P_t f \|_\infty \leq e^{-Kt} \|\nabla f \|_\infty.
 $$ 
 \item[(iii)]
 For all bounded $f \in C^{Lip}(M)$ and all $t>0$
 $$
 Lip(P_t f) \leq e^{-Kt} Lip(f).
 $$
 \item[(iv)] 
 For all $x,y \in M$ and all $t>0$
 \begin{align*}
 W(p^x_t,p^y_t) \leq e^{-Kt} d(x,y)
 \end{align*}
\end{enumerate}
where $P_t$ denotes the heat semigroup generated by the Laplace-Beltrami operator, $p_t^x$ denotes the heat kernel and $W$ denotes the $L_1$-Wasserstein distance.
\end{theorem}

We prove that the same holds true on graphs (see Theorem~\ref{thm:gradientGraphs}). 
To do so, we employ a new method intertwining the heat semigroup with a cutoff function, which we call the perpetual cutoff method.
Our curvature notion will be a generalized Ollivier curvature based on its modification by Lin, Lu and Yau (see \cite{lin2011ricci}) which we extend to the case of general graph Laplacians.  In particular, we now apply this curvature notion to graphs with unbounded vertex degree.
For an introduction to Ollivier curvature, see \cite{ollivier2007ricci, ollivier2009ricci}.  
A relation between curvature and the number of triangles in a graph is given in \cite{jost2014ollivier}.
Methods to compute the curvature can be found in \cite{loisel2014ricci}.
Ollivier curvature has been applied to describe the internet topology \cite{ni2015ricci,wang2016interference}, wireless network theory \cite{wang2014wireless}, economic market analysis \cite{sandhu2016ricci} and cancer networks \cite{sandhu2015graph,tannenbaum2015ricci,sandhu2015analytical}.

The breakthrough paper of Renesse and Sturm mentioned above motivated a generalization of Ollivier curvature to semiproups compatible with Lipschitz functions and Wasserstein metrics. Approaches in this direction have been made in \cite{bass1981markov,joulin2007poisson,joulin2009new,veysseire2012coarse}). However, the problem suggested by Ollivier (see Problem~D in \cite{ollivier2010survey}), namely, if a lower curvature bound implies non-explosion (also known as stochastic completeness), is still open.  
Non-explosion in this context means that the process remains in the state space for all time. 
We prove non-explosion for all locally finite graph Laplacians with Ollivier curvature decaying not faster than $-\log  R$ (see Theorem~\ref{thm:StochComplete}). Therefore, this theorem can be seen as an initial step to solve Ollivier's problem in a general setting. Moreover, we prove that the curvature decay rate $-\log R$ is optimal.

One of the main observations of our paper is that, on graphs, the limit expression for Lin, Lu and Yau's Ollivier curvature simplifies to the limit-free expression
\begin{align*}
\kappa(x,y) = \inf_{\substack {f\in Lip(1) \\ \nabla_{yx}f=1}}  \nabla_{xy} \Delta f
\end{align*}
where $\nabla_{xy} f = \frac{f(x)-f(y)}{d(x,y)}$, $d$ is the usual combinatorial graph distance, $Lip(1)$ denotes the functions with Lipschitz constant 1 with respect to this metric and $\Delta$ is the graph Laplacian (see Theorem~\ref{thm:nablaDelta}).
Furthermore, the curvature admits another limit-free expression in terms of transport costs (see Proposition~\ref{pro:CharTransport}) which simplifies to give an explicit formula in the case of combinatorial graph (see Theorem~\ref{thm:TransprtCombinatorial}). 
These simplifications give the starting point for proving the semigroup characterization of a lower Ricci curvature bound in analogy to the work of Renesse and Sturm.
Our gradient estimates for the continuous time heat equation seem to be the first result of this kind for Ollivier curvature.
Indeed, these gradient estimates have been previously used as a globally defined Wasserstein curvature bound (see \cite[Definition~2.1]{joulin2007poisson}).

In contrast to Ollivier curvature, there are various gradient estimates on graphs under Bakry-Emery curvature bounds \cite{horn2014volume,gong2015properties,munch2014li,lin2015equivalent,bauer2015li} and under entropic Ricci curvature \cite[Theorem~3.1]{erbar2016poincar}.
Using a modification of Ollivier curvature, gradient estimates have been established for continuous time Markov processes in \cite{veysseire2012coarse}. In Section~\ref{sec:MarkovProcesses}, we prove that this modification coincides with our curvature notion on locally finite graphs with a lower curvature bound (see Corollary~\ref{cor:MarkovChains}).

The control of the Lipschitz constant of the semigroup yields stochastic completeness
 for all graphs with a constant lower Ollivier curvature bound (see Lemma~\ref{lem:StochComplete}).  However, as already mentioned above, one can get even better results by employing different techniques.
To do so, we first establish a Laplacian comparison principle which seems to be the first of its kind for any discrete Ricci curvature notion. 
The simplest version (Theorem~\ref{thm:LaplaceCompare}) states that under the assumption of a lower curvature bound $K \in \IR$, we have
\begin{align*}
\Delta d(x_0,\cdot) \leq \Deg(x_0) -Kd(x_0,\cdot)
\end{align*}
where $\Deg(x_0)$ is the weighted vertex degree of a fixed vertex $x_0$.
This Laplacian comparison can be extended to the case of decaying curvature (see Theorem~\ref{thm:LaplaceComparisonNonConst}).
Via the Laplacian comparison, we compare the curvature of a given graph to the curvature of the birth-death chain associated to the graph (see Corollary~\ref{cor:CurvCompare}). 
Birth-death chains are Markov processes on weighted path graphs. For our purposes, we identify the birth-death chain with its associated weighted path graph, see Section~\ref{sec:OllivierBirthDeath}.
The comparison to birth-death chains allows us to reduce many problems to the case of weighted path graphs where the curvature can be easily calculated (see Theorem~\ref{thm:line}).

Using the Laplacian comparison principle and Khas'minskii's criterion (see \cite{huang2011stochastic}), we will prove stochastic completeness under the assumption that the Ollivier curvature does not decay to $-\infty$ faster than $-\log R$ (see Theorem~\ref{thm:StochComplete}).
This result is remarkable when comparing to known stochastic completeness results for graphs which use the Bakry-Emery curvature instead and require a constant lower bound
as well as additional assumptions, such as a non-local completeness condition and a lower bounded vertex measure (see \cite[Theorem~1.2]{hua2017stochastic}).
As such, this article may also give inspiration to transfer the new methods presented here to discrete Bakry-Emery theory.

 As another application of the Laplacian comparison principle, we can prove finiteness and improved diameter bounds. 
  For diameter bounds under uniformly positive Ollivier curvature, see \cite{lin2011ricci, ollivier2009ricci}. Diameter bounds under a positive average Ollivier curvature can be found in \cite{paeng2012volume}. Diameter bounds under uniformly positive Bakry-Emery curvature are proven in \cite{liu2016bakry,fathi2015curvature,horn2014volume}.
 In this article we show that if the vertex degree is bounded and the curvature decays not faster than $1/R$, then the graph is finite (see Theorem~\ref{thm:ImprovedDiamBound} and Corollary~\ref{cor:finite}). 
Surprisingly,  uniformly positive Ricci curvature alone does not imply finiteness (see Example~\ref{Ex:positiveCurvInfiniteDiam}). However, if we additionally assume a lower bound on the vertex measure, then uniformly positive Ricci curvature indeed implies finiteness (see Corollary~\ref{cor:FiniteDiamBoundedMeasure}).

\subsection{Discussion and comparison to manifolds}

The reader familiar with the manifold case might be surprised at the optimal curvature decay rates: $- \log R$ for stochastic completeness and $1/R$ for compactness in the case of graphs.

In \cite[Theorem~15.4]{grigor1999analytic}, it is shown that, for manifolds, the optimal curvature decay rate for stochastic completeness is $-R^2$ which was proven in \cite{varopoulos1983potential, hsu1989heat}.
One tempting explanation for the discrepancy of the decay rate between manifolds and graphs is the choice of the metric. Frequently, intrinsic metrics introduced in \cite{frank2014intrinsic} are used to describe the geometry of graphs with unbounded vertex degrees and to give analogues to results on manifolds (see, for example, \cite{bauer2012cheeger,huang2013note,folz2015volume,keller2015intrinsic}).  
However, we give an example of a stochastically incomplete graph with curvature decaying like $-(\log \sigma)^{1+\eps}$ with respect to an intrinsic metric $\sigma$, even if the curvature is defined by incorporating the intrinsic metric, see Example~\ref{ex:incompleteIntrinsic}.

The optimal decay rate on manifolds to guarantee compactness is $C/R^2$ for some constant $C$. Interestingly, for $C>\frac{n-1}4$, compactness holds, but for $C \leq \frac{n-1}4$, non-compact manifolds are known with the corresponding Ricci curvature decay (see \cite{cheeger1982finite,holcman2005boundary}).
The discrepancy of the decay rate between manifolds and graphs here also cannot be explained via intrinsic metrics since we assume a bounded vertex degree for our result and, therefore, the combinatorial distance is intrinsic up to a factor. Hence, it might be interesting to ask for the deeper reasons for these two discrepancies.

Before introducing the setup and notations, we give a brief summary of the subsequent sections.
In Section~\ref{sec:Ollivier}, we prove the limit-free simplifications of the curvature formula and compute the curvature of combinatorial graphs and birth-death chains. 
In Section~\ref{sec:GradEstimates}, we introduce the perpetual cutoff method and non-linear cutoff semigroups which turn out to perfectly intermesh with Ollivier curvature, yielding the desired gradient estimate for the semigroup.
In Section~\ref{sec:LaplaceCompare}, we present the Laplacian comparison theorem and, as applications, we prove a birth-death chain reduction, stochastic completeness and improved diameter bounds which lead to our finiteness results.
Finally in Section~\ref{sec:MarkovProcesses}, we prove that on graphs with a lower curvature bound, our curvature notion coincides with the curvature introduced in \cite{veysseire2012coarse}.

\subsection{Setup and Notation}

A triple $G=(V,w,m)$ is called a \emph{graph} if $V$ is a countable set, $w:V^2 \to [0,\infty)$ is symmetric and zero on the diagonal and $m:V \to (0,\infty)$. In the following, we only consider \emph{locally finite} graphs, i.e., for every $x \in V$ there are only finitely many $y \in V$ with $w(x,y) >0$.
We call $V$ the \emph{vertex set} with elements of $V$ called \emph{vertices}, $w$ the \emph{edge weight} and $m$ the \emph{vertex measure}. 
We will write $x \sim y$ if $w(x,y)>0$ and say that $(x,y)$ is an \emph{edge} in the graph.
We say that $G$ is a \emph{combinatorial graph} if $w(x,y) \in \{0,1\}$ for all $x,y \in V$ and if $m \equiv 1$.

We define the \emph{graph Laplacian} $\Delta: \IR^V \to \IR^V$ via
$$ \Delta f(x) := \frac 1 {m(x)} \sum_{y\in V} w(x,y)(f(y) - f(x)). $$ 
We define the function spaces
\begin{align*}
C(V)&:=\{f:V \to \IR\}= \IR^V, \\
\ell_\infty(V)&:=\{f \in C(V): f \mbox{ is bounded}\}, \\
C_c(V) &:= \{f \in C(V): f \mbox{ is finitely supported}\},
\end{align*}
all endowed with the supremum norm $\|\cdot \|_\infty$. 
We let $\ell^p(V,m)$ denote the $\ell^p$ spaces with respect to $m$, that is,
$\ell^p(V,m) = \{ f \in C(V) : \sum_{x \in V} |f(x)|^p m(x) <\infty\}.$ 
We let 
$$\Deg(x) := \frac1  {m(x)} \sum_{y \in V} {w(x,y)}$$ 
denote the \emph{vertex degree} and let $\Deg_{\max} := \sup_x \Deg(x) \in (0,\infty]$. We remark that the Laplace operator is bounded on $\ell_\infty(V)$ and $\ell^p(V,m)$ for any $p\geq 1$ if and only if $\Deg_{\max} < \infty$  (see \cite[Theorem 11]{keller2010unbounded}, \cite[Theorem 9.3]{haeseler2011laplacians}).
In this case, we will say that $G$ has \emph{bounded vertex degree}.

For a non-negative $f \in \ell_\infty(V)$, we denote by $P_t f$ the smallest  non-negative bounded continuous solution $u(x,t)$ to the heat equation 
\begin{align*}
\begin{cases}
\begin{tabular}{rll}
$\Delta u(x,t)$  & $=\partial_t u(x,t)$ & $x \in V$, $t\geq0$ \\
$u(x,0)$  & $= f(x)$ &$x \in V$.
\end{tabular}
\end{cases}
\end{align*}
A proof of the existence and uniqueness of $P_t f$ and further details can be found in \cite{wojciechowski2008heat, keller2012dirichlet}.  Note, in particular, that $P_{s+t} f = P_s P_t f$ which is referred to as the \emph{semigroup property} and that $P_t$ is \emph{positivity preserving}, i.e., $P_t f\geq0$ for $f\geq0$.  
A graph is called \emph{stochastically complete} or \emph{non-explosive} if $P_t \mathbf 1 =  \mathbf 1$  for all $t>0$ where $\mathbf 1$ is the constant  function which is 1 on $V$. 

We define the \emph{combinatorial graph distance} $d$ on $V\times V$ via
$d(x,y):= \inf\{n:x=x_0\sim \ldots \sim x_n=y\}$. A graph is said to be \emph{connected} if $d(x,y) < \infty$ for all $x,y$ in $V$. We will always assume that graphs are connected.  We write $B_r(x) = \{ y \in V \ | \ d(x,y) \leq r \}$ and $S_r(x) = \{ y \in V \ | \ d(x,y)=r \}$.
We note that $G$ is connected if and only if $P_t$ is a positivity improving semigroup, that is, $P_t f >0$ if $f \geq 0$ for $t>0$, see \cite{keller2012dirichlet}.

We write $f \in Lip(1)$ if $|f(x)-f(y)| \leq d(x,y)$ for all $x,y \in V$.
The \emph{Wasserstein distance} $W(\mu,\nu)$ for probability measures $\mu$ and $\nu$ on $V$ is given by
\begin{align*}
W(\mu,\nu) &:= \sup_{f \in Lip(1) \cap \ell_\infty(V)} \int f d\mu - \int f d\nu \\
&= \sup_{f \in Lip(1) \cap \ell_\infty(V)} \sum_{x \in V} f(x) (\mu(x) - \nu(x)).
\end{align*}
We note that the supremum is well defined due to the boundedness of the functions and that it suffices to take the supremum over functions in $Lip(1)$ when the measures are finitely supported.
Equivalently (see e.g.  \cite[Theorem 1.14]{villani2003topics}), the Wasserstein metric can be defined as 
$$
W(\mu,\nu) := \inf_{\rho} \sum_{x,y \in V} \rho(x,y) d(x,y)
$$
where the infimum is taken over all $\rho: V^2 \to [0,1]$ which satisfy $\sum_{y \in V} \rho(x,y) = \mu(x)$ and $\sum_{x \in V} \rho(x,y) = \nu(y)$ for all $x,y \in V$.  We call such a $\rho$ a \emph{coupling} between $\mu$ and $\nu$.

\section{Ollivier curvature and graph Laplacians}\label{sec:Ollivier}
Ollivier curvature is a powerful and easy to calculate tool used to study analytic and geometric properties of a space. Until now, Ollivier curvature for graphs has only been used in the case of bounded Laplacians. In this section, we extend the definition of Ollivier curvature to the case of unbounded graph Laplacians.  We then present a strikingly simple expression for calculating the curvature using only the Laplacian as well as a formula involving transport costs.  Along the way, we illustrate how to calculate the curvature in a variety of situations including graphs without cycles, combinatorial graphs and birth-death chains.

\bigskip

For $\eps>0$, we let
\begin{align*}
m_x^\eps(y) := 1_y(x) + \eps \Delta 1_y (x)
\end{align*}
which is a finitely supported probability measure and, in particular, non-negative if $\eps$ is sufficiently small. 
This can be seen as
\begin{align*}
m_x^\eps (y) = \begin{cases}
1 - \eps \Deg(x) &: y=x \\
\eps w(x,y)/m(x) &: \mbox{otherwise}.
\end{cases}
\end{align*}
In particular,
\begin{align*}
\int f dm_x^\eps = \sum_{y \in V} f(y) m_x^\eps(y) = (f + \eps \Delta f)  (x).
\end{align*}
We remark that $m_x^\eps$ can be seen as a first order approximation to the heat kernel $P_\eps 1_x$.  This connection will be further explored in Section~\ref{sec:MarkovProcesses}.

In the case of the normalized Laplacian, that is, when $w: V^2 \to \{ 0,1\}$ and $m(x) = d_x := \# \{y \sim x \}$, for $\alpha := 1-\eps$ one has
\begin{align*}
m_x^\eps (y) = \begin{cases}
\alpha &: y=x \\
(1-\alpha)/d_x&:  y \sim x \\
0 &: \mbox{otherwise}
\end{cases}
\end{align*}
which corresponds to the definition of Lin, Lu and Yau (see \cite{bauer2011ollivier,lin2011ricci}).  Note that, in this case, $\Deg=1$ so that the normalized Laplacian is always a bounded operator.

Following the standard definition, we let, for $x \not = y$
\begin{align*}
\kappa_\eps(x,y) := 1 - \frac{W(m_x^\eps,m_y^\eps)}{d(x,y)}
\end{align*}
where $W$ denotes the Wasserstein distance.
In \cite{bourne2017ollivier} it is shown that for the normalized Laplacian and $x \sim y$, the function $\kappa_\eps(x,y)$ is concave  and piecewise linear in $\eps \in [0,1]$ with at most three linear parts.
Analogous to Lin, Lu and Yau, one can prove the existence of
\begin{align*}
\kappa(x,y) := \lim_{\eps \to 0^+} \frac 1 \eps \kappa_\eps(x,y)
\end{align*}
by which we extend Lin, Lu and Yau's curvature definition to arbitrary graph Laplacians.

Using standard theory, the curvature $\kappa(x,y)$ is uniquely determined by the induced subgraph $B_1(x) \cup B_1(y)$ for $x \sim y$.
We write $Ric(G) \geq K$ if $\kappa(x,y) \geq K$ for all $x,y$.
We remark that to show $Ric(G) \geq K$ it suffices to show that $\kappa(x,y) \geq K$ for adjacent vertices $x \sim y$ as in \cite{lin2011ricci}.

As a first example, we mention that it is well-known that, in the case of the normalized Laplacian, Abelian Cayley graphs have non-negative Ollivier curvature (see e.g. \cite[Theorem~2]{lin2014ricci}).  We will give further examples later in this section.

\subsection{Bypassing the limit}

Our first aim is to express the curvature without the limit which turns out to be surprisingly simple. To do so, we introduce the notation of the gradient

$$
\nabla_{xy} f := 
\frac{f(x) - f(y)}{d(x,y)} 
$$

for $x\neq y \in V$ and $f \in C(V)$ and the associated Lipschitz constant
$$
\|\nabla f \|_\infty := \sup_{x \neq y} |\nabla_{xy} f| = \sup_{x \sim y} |\nabla_{xy} f| \in [0,\infty].
$$
For $K \geq 0$, we let $Lip(K) = \{f \in C(V):  \|\nabla f \|_\infty \leq K \},$ that is, the set of functions with Lipschitz constant $K$ or $K$-Lipschitz functions.
We are now prepared to present our limit-free curvature formula.
\begin{theorem}[Curvature via the Laplacian]\label{thm:nablaDelta}
Let $G=(V,w,m)$ be a graph and let $x \neq y$ be vertices. Then,
\begin{align*}
\kappa(x,y) = \inf_{\substack {f\in Lip(1) \cap C_c(V) \\ \nabla_{yx}f=1}}  \nabla_{xy} \Delta f.
\end{align*}
\end{theorem}

\begin{proof}
	By definition, one has

\begin{align*}
W(m_x^\eps,m_y^\eps) &=  \sup_{f \in Lip(1)} \sum_z f(z)(m_y^\eps(z) - m_x^\eps(z))\\ 
&= \sup_{f \in Lip(1)} [(f(y) + \eps \Delta f(y)) - (f(x) + \eps \Delta f(x))]\\
&= d(x,y)\sup_{f \in Lip(1)} \nabla_{yx}(f + \eps \Delta f).
\end{align*}
Hence,
\begin{align*}
\frac 1 \eps   \kappa_\eps(x,y) 
&= \frac 1 \eps  \left( 1- \frac{W(m_x^\eps,m_y^\eps)}{d(x,y)} \right) \\
&=  \frac 1 \eps  \left( \inf_{f \in Lip(1)} ( 1- \nabla_{yx}(f + \eps \Delta f) ) \right)\\
& = \inf_{f \in Lip(1)}  \left(\frac 1 \eps(1- \nabla_{yx}f)    +\nabla_{xy} \Delta f \right)\\
&\leq  \inf_{\substack {f\in Lip(1) \cap C_c(V) \\ \nabla_{yx}f=1}}  \nabla_{xy} \Delta f 
\end{align*}

To prove the other inequality, we first show the existence of a minimizer $f_\eps \in Lip(1) \cap C_c(V)$ of the expression $\frac 1 \eps(1- \nabla_{yx}f)    +\nabla_{xy} \Delta f$ found above for every $\eps >0$ satisfying $f_\eps(x)=0$.
This follows as, for every $f \in Lip(1)$, we construct $\widetilde f \in Lip(1)$ supported on $B_{2r}(x)$ with $r:=d(x,y)+1$ which satisfies 
\begin{align}
\frac 1 \eps(1- \nabla_{yx}f)    +\nabla_{xy} \Delta f = \frac 1 \eps(1- \nabla_{yx}\widetilde f)    +\nabla_{xy} \Delta \widetilde f. \label{eq:fwidetildefMaxExist}
\end{align}
 By adding a constant to $f$, we can assume that $f(x)=0$.
This yields that $|f(z)|\leq r$ for all $z \in B_1(x) \cup B_1(y)$ since $f \in Lip(1)$.
Let $\phi:V \to \IR$ be given by 
\[
\phi(z) =  \left[r \wedge (2r - d(x,z)) \right]_+.
\] 
Observe that $\phi(z)=r$ for all $z \in B_1(x)\cup B_1(y)$.
Therefore,
$\widetilde f := -\phi \vee f \wedge \phi$
satisfies (\ref{eq:fwidetildefMaxExist}) as it agrees with $f$ on  $B_1(x)\cup B_1(y)$. Moreover, $\phi$ and thus $\widetilde f$ are supported on $B_{2r}(x)$.
This construction of $\widetilde f$ shows that we can restrict the infimum to functions supported on the compact set $B_{2r}(x)$ which yields the existence of a minimizer $f_\eps$ with $f_\eps(x)=0$ for all $\eps>0$ due to continuity.

Due to the compactness of $B_{2r}(x)$ and since $f_\eps(x)=0$ and $f_\eps \in Lip(1)$ for all $\eps>0$, there exists a sequence $\eps_n \to 0$ such that $f_0 := \lim_{n\to \infty} f_{\eps_n}$ exists.
Since $\frac 1 \eps \kappa_\eps(x,y) = \frac{1}{\eps} (1-\nabla_{yx}f_\eps)+\nabla_{xy}\Delta f_\eps$ and $\lim_{\eps \to 0^+} \frac 1 \eps \kappa_\eps(x,y)$ exists,
we get that $\nabla_{yx}f_\eps \to 1$ as $\eps \to 0^+$.
Therefore, $f_0 \in Lip(1) \cap C_c(V)$,  $\nabla_{yx}f_0 = 1$ and since $\nabla_{yx} f_\eps \leq 1$, we get 
\begin{align*}
\kappa(x,y) &= \lim_{\eps \to 0^+}   \frac 1 \eps(1- \nabla_{yx}f_\eps)    +\nabla_{xy} \Delta f_\eps \\
&\geq \lim_{n \to \infty} \nabla_{xy} \Delta f_{\eps_n} \\
& = \nabla_{xy} \Delta  f_0\\
&\geq \inf_{\substack {f\in Lip(1) \cap C_c(V) \\ \nabla_{yx}f=1}}  \nabla_{xy} \Delta f.
\end{align*} 

Putting together the upper and lower estimates yields
\begin{align*}
\kappa(x,y) = \inf_{\substack {f\in Lip(1) \cap C_c(V) \\ \nabla_{yx}f=1}}  \nabla_{xy} \Delta f
\end{align*}
as desired.	
\end{proof}

Following \cite[Lemma~2.2]{bhattacharya2015exact}, it suffices to optimize over all integer valued Lipschitz functions $f$ which yields the following corollary.

\begin{corollary}
Let $G=(V,w,m)$ be a graph and let $x \neq y$ be vertices. Then,
\begin{align*}
\kappa(x,y) = \inf_{\substack {f: B_1(x) \cup B_1(y) \to \IZ \\ f\in Lip(1)  \\ \nabla_{yx}f=1}}  \nabla_{xy} \Delta f
\end{align*}
Moreover, on combinatorial graphs, the curvature $\kappa(x,y)$ is integer valued for all $x\sim y$.
\end{corollary}
\begin{proof}
The proof of the first part follows \cite[Lemma~2.2]{bhattacharya2015exact}. For the integrality of the curvature in the case of combinatorial graphs, observe that $\nabla_{xy}\Delta f \in \IZ$ whenever $f$ is integer valued, $x\sim y$ and $\Delta$ is the combinatorial graph Laplacian.
\end{proof}
We now explicitly calculate the curvature of large girth graphs in our setting by using Theorem~\ref{thm:nablaDelta}. 
\begin{example}\label{ex:NoCycles}
Let $G=(V,w,m)$ be a graph and let $x\sim y$ be vertices.
Suppose that the edge $(x,y)$ is not contained in any 3-,4- or 5-cycles. Then, an optimal 1-Lipschitz function $f$ is given by an extension of
\[
f(z) = \begin{cases}
0 &: z \sim x \mbox{ and } z \neq y\\
1 &: z=x\\
2 &: z=y\\
3 &: z \sim y \mbox{ and } z \neq x
\end{cases}
\]
yielding the curvature
\[
\kappa(x,y) = 2w(x,y) \left(\frac 1{m(x)} + \frac 1 {m(y)} \right) - \Deg(x) - \Deg(y).
\]
\end{example}

\bigskip

We now give another limit-free expression of our extension of Lin-Lu-Yau's Ollivier curvature via transport costs.

\begin{proposition}[Curvature via transport cost]\label{pro:CharTransport}
Let $G=(V,w,m)$ be a graph and let $x_0 \neq y_0$ be vertices. Then,
\begin{align}
\kappa(x_0,y_0) &= \sup_{\rho}  \sum_{\substack{x \in B_1(x_0) \\ y \in B_1(y_0)}}\rho(x,y) \left[1 - \frac{d(x,y)}{d(x_0,y_0)}\right] \label{eq:PropTransport} 
\end{align}
where the supremum is taken over all $\rho: B_1(x_0) \times B_1(y_0) \to [0,\infty)$ such that
\begin{align}
\sum_{y \in B_1(y_0)} \rho(x,y) &= \frac {w(x_0,x)}{m(x_0)}  \qquad \mbox{ for all } x \in S_1(x_0) \mbox{ and} \label{eq:rhoXProp}\\
\sum_{x \in B_1(x_0)} \rho(x,y) &= \frac {w(y_0,y)}{m(y_0)} \qquad \mbox{ for all } y \in S_1(y_0)  \label{eq:rhoYProp}.
\end{align}
\end{proposition}
\begin{rem}
We remark that $\rho$ is defined on balls, but we only require the coupling property on spheres.  We additionally do not assume anything concerning
$\sum_{x,y} \rho(x,y)$.
\end{rem}
\begin{proof}
We will write
\[ F(\rho) = \sum_{x \in B_1(x_0)} \sum_{y \in B_1(y_0)} \rho(x,y) \left[1 - \frac{d(x,y)}{d(x_0,y_0)}\right] \]
for any coupling $\rho$.   We wish to show that $\kappa(x_0,y_0) = \sup_\rho F(\rho)$ where the supremum is taken over all couplings $\rho$ satisfying (\ref{eq:rhoXProp}) and (\ref{eq:rhoYProp}).

Using that $\sum_{x,y}\rho(x,y)=1$ for all couplings considered in the transport definition of $W$, we have
	\begin{align*}
     \kappa_\eps(x_0,y_0) =  1 - 	\frac {W(m_{x_0}^\eps,m_{y_0}^\eps)}{d(x_0,y_0)} = 1 - \frac{\inf_\rho \sum_{x,y} \rho(x,y)d(x,y)}{d(x_0,y_0)} = \sup_\rho F(\rho)	
	\end{align*}
	where the supremum is taken over all  $\rho: B_1(x_0) \times B_1(y_0) \to [0,\infty)$ such that
	\begin{align*}
		\sum_{y \in B_1(y_0)} \rho(x,y) &= m_{x_0}^\eps(x) = 1_x(x_0) + \eps \Delta 1_x(x_0) \qquad  \mbox{ for all } x \in B_1(x_0) \mbox{ and} \\
		\sum_{x \in B_1(x_0)} \rho(x,y) &= m_{y_0}^\eps(y) = 1_y(y_0) + \eps \Delta 1_y(y_0) \qquad  \mbox{ for all } y \in B_1(y_0). 
	\end{align*}	
	Dividing $\rho$ by $\eps$ yields	
	\begin{align*}
	\frac{1}{\eps} \kappa_\eps(x_0,y_0) = \frac 1 {\eps} \left(1 - 	\frac {W(m_{x_0}^\eps,m_{y_0}^\eps)}{d(x_0,y_0)} \right) =
	\sup_\rho F(\rho)
	\end{align*}
	with the supremum taken over all  $\rho: B_1(x_0) \times B_1(y_0) \to [0,\infty)$ such that
	\begin{align}
	\sum_{y \in B_1(y_0)} \rho(x,y) &= \frac 1 \eps 1_x(x_0) + \Delta 1_x(x_0)  \qquad \mbox{ for all } x \in B_1(x_0) \mbox{ and}  \label{eq:rhoXProof} \\
	\sum_{x \in B_1(x_0)} \rho(x,y) &= \frac 1 \eps 1_y(y_0) + \Delta 1_y(y_0) \qquad \mbox{ for all } y \in B_1(y_0) \label{eq:rhoYProof}.
	\end{align}	
	
	We remark that (\ref{eq:rhoXProp}) and (\ref{eq:rhoYProp}) hold for all $\rho$ satisfying (\ref{eq:rhoXProof}) and (\ref{eq:rhoYProof}) as  
	$\Delta 1_x (x_0)= \frac {w(x_0,x)}{m(x_0)}$ for $x \not = x_0$.
	Therefore, $\frac 1 \eps \kappa_\eps(x_0,y_0)$ is less than or equal to the right hand side of (\ref{eq:PropTransport}).
	
	We now show that if we modify $\rho$ satisfying (\ref{eq:rhoXProp}) and (\ref{eq:rhoYProp}) appropriately, then the value of $F(\rho)$ in the right hand side of  (\ref{eq:PropTransport}) does not change and the modified $\rho$ satisfies  (\ref{eq:rhoXProof}) and (\ref{eq:rhoYProof}) which will show that $\frac 1 \eps \kappa_\eps(x_0,y_0)$ is larger than or equal to the right hand side of (\ref{eq:PropTransport}) for small $\eps$.
	
	Suppose that $\rho$ satisfies (\ref{eq:rhoXProp}) and (\ref{eq:rhoYProp}).
	We define
	\begin{align*}
	\rho_\eps(x,y) := \rho(x,y) + 1_{x}(x_0)1_{y}(y_0) \left( \frac 1 \eps - \sum_{u,v} \rho(u,v) \right)
	\end{align*}
	which is non-negative if $\eps$ is small.
	We observe that
		\begin{align*}
		F(\rho)= \sum_{x,y} \rho(x,y) \left[1 - \frac{d(x,y)}{d(x_0,y_0)} \right] = \sum_{x,y} \rho_\eps(x,y) \left[1 - \frac{d(x,y)}{d(x_0,y_0)} \right] = F(\rho_\eps)
		\end{align*}
		since $\rho_\eps(x,y)$ and $\rho(x,y)$ only differ at $(x_0,y_0)$ where the latter factor in the sums vanishes.
	Moreover,	
	\begin{align*}
	\sum_{x \in B_1(x_0)} \sum_{y \in B_1(y_0)} \rho_\eps(x,y) = \frac 1 \eps.
	\end{align*}
	
	We now show that $\rho_\eps$ satisfies (\ref{eq:rhoXProof}).
	Since $\frac 1 \eps 1_x(x_0) = 0$ on $S_1(x_0)$, we see that (\ref{eq:rhoXProp}) implies (\ref{eq:rhoXProof}) for $x\in S_1(x_0)$.
	For the remaining case $x=x_0$, equation (\ref{eq:rhoXProof}) follows since by (\ref{eq:rhoXProp}),
	\begin{align*}
	\sum_{y\in B_1(y_0)} \rho_\eps(x_0,y) = \frac 1 \eps - \sum_{x \in S_1(x_0)} \sum_{y \in B_1(y_0)} \rho(x,y) = \frac 1 \eps - \sum_{x \in S_1(x_0)} \frac{w(x_0,x)}{m(x_0)}	= \frac 1 \eps + \Delta 1_{x_0} (x_0).
	\end{align*}
Due to an analogous argument, $\rho_\eps$ also satisfies (\ref{eq:rhoYProof}).
Putting everything together proves that $\frac 1 \eps \kappa_\eps(x_0,y_0)$ equals the right hand side of (\ref{eq:PropTransport}) for small $\eps$.	
Taking $\eps \to 0^+$ finishes the proof.
\end{proof}

\subsection{Ollivier curvature on combinatorial graphs}

We now show how the transport cost expression for the curvature simplifies on combinatorial graphs.
We first describe the curvature on combinatorial graphs intuitively.  We note how this case complements Example~\ref{ex:NoCycles} 
which considered the case of graphs with no cycles.
\begin{itemize}
\item
Given an edge $x \sim y$, we have initial curvature $\kappa(x,y)=2$.
\item
Every triangle containing $x,y$ increases $\kappa(x,y)$ by one.
\item
Adding 4-cycles containing $x,y$ does not change $\kappa(x,y)$.
\item
Adding 5-cycles containing $x,y$ decreases $\kappa(x,y)$ by one.
\item
Every additional neighbor of both $x$ and $y$ decreases $\kappa(x,y)$ by one.
\end{itemize}
The following theorem gives a precise expression for the curvature of combinatorial graphs making the above intuition explicit.

\begin{theorem}[Transport and combinatorial graphs]\label{thm:TransprtCombinatorial}
Let $G=(V,w,m)$ be a combinatorial graph and let $x_0 \sim y_0$ be adjacent vertices. 
Let $B_{x_0y_0}:= B_1(x_0) \cap B_1(y_0)$, $B_{x_0}^{y_0}:=B_1(x_0) \setminus B_1(y_0)$ and  $B_{y_0}^{x_0}:=B_1(y_0) \setminus B_1(x_0).$
Let $\Phi_{x_0y_0} := \{\phi: D(\phi) \subseteq B_{x_0}^{y_0} \to R(\phi) \subseteq B_{y_0}^{x_0} :  \phi \mbox{ bijective}\}. $
 For $\phi \in \Phi_{x_0y_0}$ write  $D(\phi)^c := B_{x_0}^{y_0} \setminus D(\phi)$ and
$R(\phi)^c :=  B_{y_0}^{x_0} \setminus R(\phi).$
Then,
\begin{align*}
\kappa(x_0,y_0) = \# B_{x_0y_0} - \inf_{\phi \in \Phi_{x_0y_0}} \left( \#D(\phi)^c + \# R(\phi)^c + \sum_{x\in D(\phi)}   [d(x,\phi(x)) - 1]  \right).
\end{align*}
\end{theorem}
We remark that $\Phi_{x_0y_0} \neq \emptyset$ since $\Phi_{x_0y_0}$ always contains the empty function.

\begin{figure}[h]
	\centering
\begin{tikzpicture}[scale=0.8, transform shape]

\vertex[label={[name=lx] below:{$x_0$}}](x)at (-1,-1) {};
\vertex[label={[name=ly] below:{$y_0$}}](y) at (1,-1) {};
\node(common)[draw,rectangle] at (0,0.5) {$S_1(x_0) \cap S_1(y_0)$};
\Edge(x)(y)
\Edge(x)(common)
\Edge(y)(common)

\node[draw, circle, fit=(x)(y)(common)(lx)(ly),minimum size=4cm, label=below:{$B_{x_0y_0}$}](BI){};


\node[draw, rectangle](D) at 	(-3,2)		{$D(\phi)$};
\node(Dc) at 				(-3,3)		{$D(\phi)^c$};
\node[label=left:{$B_{x_0}^{y_0}$}, fit=(D)(Dc)](Bxy)[draw, circle, minimum size=2.8cm] {};

\node[draw, rectangle](R) at 	(3,2)			{$R(\phi)$};
\node(Rc) at 				(3,3)			{$R(\phi)^c$};
\node[label=right:{$B_{y_0}^{x_0}$}, fit=(R)(Rc)](Byx)[draw, circle, minimum size=2.8cm] {};

\Edge(x)(y)
\Edge(x)(common)
\Edge(y)(common)
\Edge(x)(Bxy)
\Edge(y)(Byx)

\draw[->,>=stealth,line width = 1.5pt] (D) -- node[label=$\phi$] {}(R);
				
\end{tikzpicture}
\caption{The figure is a scheme of  the terms used in Theorem~\ref{thm:TransprtCombinatorial} providing a simple method to compute the curvature for combinatorial graphs.}
\label{fig:transportCombinatorial}
\end{figure}
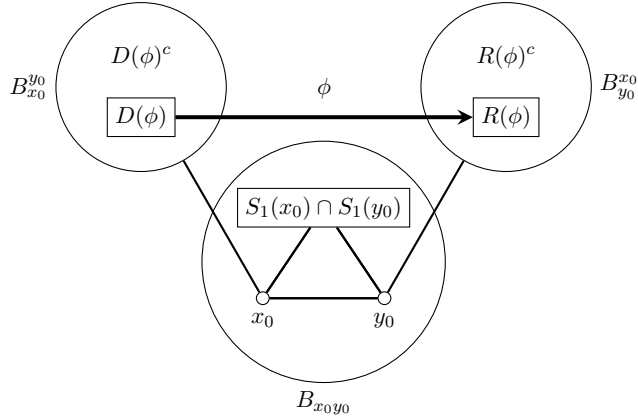
To prove the theorem, we first show that the coupling function $\rho$ which gives $\kappa(x_0,y_0)$ via the expression found in Proposition~\ref{pro:CharTransport} can be assumed to be integer valued for combinatorial graphs.  As in the proof of Proposition~\ref{pro:CharTransport} and since we assume that $x_0 \sim y_0$, we let 
\[ F(\rho)= \sum_{x,y} \rho(x,y)(1 - d(x,y)) \]
for $\rho: B_1(x_0) \times B_1(y_0) \to [0,\infty)$.
We note that in the case of combinatorial graphs, (\ref{eq:rhoXProp}) and (\ref{eq:rhoYProp}) become
\begin{equation}\label{eq:TransportCombinatorial}
\sum_{y \in B_1(y_0)} \rho(x,y) = 1 \quad \forall x \sim x_0 \qquad \mbox{ and}  \qquad \sum_{x \in B_1(x_0)} \rho(x,y) = 1 \quad \forall y \sim y_0. 
\end{equation}
In particular, as we assume that $x_0 \sim y_0$, we have $\sum_x \rho(x,x_0)= \sum_y \rho(y_0,y) =1$.
\begin{lemma}\label{lem:01lemma}
Let $G=(V,w,m)$ be a combinatorial graph and let $x_0 \sim y_0$. Then, there exists $\rho:B_1(x_0) \times B_1(y_0) \to \{0,1\}$ satisfying \eqref{eq:TransportCombinatorial}   such that
$\kappa(x_0,y_0) = F(\rho)$.
Furthermore, $\rho$ can be chosen to satisfy $\rho(z,z)=1$ for all $z \in B_{x_0y_0}$.
\end{lemma}

\begin{proof}
We first show that $\rho$ can be chosen to take values in $\{0,1\}$.
Suppose not. Let $\rho$ be a coupling which satisfies \eqref{eq:TransportCombinatorial} such that $\kappa(x_0,y_0) = F(\rho)$ and so that $\rho$ has the minimal number of non-$\{0,1\}$ entries.
Denote by 
\[ M=\{(x,y) \in B_1(x_0) \times B_1(y_0):\rho(x,y) \notin \{0,1\}\}. \] 
By assumption $M \not = \emptyset$.
We first note that $(x_0,y_0) \not \in M$ as, if $(x_0,y_0) \in M$, then we could replace $\rho$ by a coupling whose value at $(x_0,y_0)$ is 0 without changing the value of $F(\rho)$.

By using \eqref{eq:TransportCombinatorial} repeatedly, 
we can then construct a maximal sequence $S=((x_A,y_A),\ldots,(x_B,y_B))$ in $M$ with $B\geq A \geq 0$ which has the following properties:
\begin{enumerate}
\item
$x_{2n+1} = x_{2n} \neq x_k$ for all $k \notin \{2n,2n+1\}$.
\item
$y_{2n}=y_{2n-1} \neq y_k$ for all $k \notin \{2n,2n-1\}$.
\end{enumerate}
Without loss of generality, we may assume that either $x_B=x_{B-1}$ or $A=B$.

Now, suppose that $y_B \neq y_0$.
Then, by \eqref{eq:TransportCombinatorial} there exists $(x_{B+1},y_B) \in M$ with $x_{B+1} \neq x_B$ since $\sum_x \rho(x,y_B) = 1$.
Due to the maximality of $S$, we cannot add $(x_{B+1},y_B)$ to $S$ and the only possible reason for this is that there exists $A'<B$ with $x_{A'} = x_{B+1}$ where we choose $A'$ to be maximal.
In this case, we replace $S$ by the loop $L=((x_{A'},y_{A'}), \ldots, (x_{B},y_B),(x_{B+1},y_B))$.
We proceed analogously if $y_B=y_0$ and $x_A \neq x_0$ and replace $S$ by the loop $L=((x_{A-1},y_A),\ldots,(x_{B'},y_{B'}))$.
In case we do not replace $S$ by a loop, the sequence starts with $(x_0,y_A)$ or $(x_A,y_0)$ and ends with $(x_0,y_B)$ or $(x_B,y_0)$.

Given a sequence $S$ or a loop $L$ as constructed above, we can change $\rho$ on $S$ or $L$ while preserving \eqref{eq:TransportCombinatorial}. We do this by letting $\rho_C(x_n,y_n) := \rho(x_n,y_n) + C(-1)^n$ and $\rho_C(x,y) := \rho(x,y)$ otherwise.
It is easy to check that \eqref{eq:TransportCombinatorial} also holds for $\rho_C$.
The objective function $F$ is linear. Therefore, $F(\rho_C) \geq F(\rho)$ for all negative or all positive $C$. Without loss of generality, we assume that $F(\rho_C) \geq F(\rho)$ for all positive $C$.
We choose $C$ maximal such that $\rho_C \geq 0$. Then, there exists $(x,y)$ in the sequence with $\rho_C(x,y) = \rho(x,y)-C = 0$ so that $\rho_C(x,y) \in \{0,1\}$ but $\rho(x,y) \notin \{0,1\}$.
This contradicts the minimality of the number of $\{0,1\}$ entries of $\rho$.
The contradiction finishes the proof of the first part of the statement.

We now show the furthermore statement, that is, that $\rho$ can additionally be chosen so that $\rho(z,z) =1$ for all $z \in B_{x_0y_0}=B_1(x_0) \cap B_1(y_0)$.  Suppose that $\rho(z,z) \neq 1$ for some $z \in B_{x_0y_0}$. Then, $\rho(z,z)=0$.

Case 1. 
We first assume that $z \in S_1(x_0) \cap S_1(y_0)$.
Then, there exists $x_z \in B_1(x_0)$ and $y_z \in B_1(y_0)$ with $\rho(x_z,z)=1=\rho(z,y_z)$ and thus $\rho(x_z,y_z)=0$ by \eqref{eq:TransportCombinatorial}.
Define $\widetilde\rho(z,z)=\widetilde\rho (x_z,y_z)=1$ and $\widetilde \rho(x_z,z) = \widetilde \rho (z,y_z)=0$ and $\widetilde\rho(x,y) = \rho(x,y)$ otherwise.
Then, $\widetilde \rho$ also satisfies \eqref{eq:TransportCombinatorial}.
Moreover, $F(\tilde \rho) = F(\rho) + d(x_z,z) + d(z,y_z) - d(x_z,y_z) \geq F(\rho)$.

Case 2.
If $z=x_0$, there exists $x_z\sim x_0$ with $\rho(x_z,x_0)=1$.
Now, set $\widetilde \rho(x_0,x_0)=\widetilde \rho(x_z,y_0) = 1$ and $\widetilde \rho(x_z,x_0)=0$ and $\widetilde{\rho}(x,y)=\rho(x,y)$ otherwise.
Then, $\widetilde \rho$ also satisfies \eqref{eq:TransportCombinatorial}.
Moreover, $F(\tilde \rho) = F(\rho) + 2 - d(x_z,y_0) \geq F(\rho)$.  An analogous argument works in the case $z=y_0$.

Therefore, in both cases, $\widetilde \rho$ is also a $\{0,1\}$-valued function satisfying \eqref{eq:TransportCombinatorial} such that $\kappa(x_0,y_0) = F(\widetilde \rho)$ and $\widetilde{\rho}(z,z)=0$. Repeating the argument yields the existence of a $\widetilde \rho$ such that
$\widetilde\rho(z,z) = 1$ for all $z \in B_{x_0y_0}$.
\end{proof}
\begin{rem}
One can also prove the integrality of the transport function $\rho$ in  Lemma~\ref{lem:01lemma} by using the theory of linear programming. In particular, the constraint matrix is a submatrix of the constraint matrix of a classical assignment problem and, therefore, totally unimodular. By standard theory and due to the integrality of all parameters, this implies the existence of an integral optimal solution $\rho$.
\end{rem}

We are now prepared to prove Theorem~\ref{thm:TransprtCombinatorial} expressing the curvature for combinatorial graphs via transport costs.
\begin{proof}[Proof of Theorem~\ref{thm:TransprtCombinatorial}]
Due to Lemma~\ref{lem:01lemma}, we can assume that the optimizing function $\rho$ satisfying \eqref{eq:TransportCombinatorial} and $\kappa(x_0,y_0)=F(\rho)=\sum_{x,y} \rho(x,y)(1-d(x,y))$ takes values in $\{0,1\}$ and satisfies
$\rho(z,z) = 1$ for all $z \in B_{x_0y_0}$.  Therefore, $\rho(x,x_0)=0$ for all $x \sim x_0$, $\rho(y_0,y)=0$ for all $y \sim y_0$ and $\rho(z,y)=\rho(x,z)=0$ for all $z \in S_1(x_0) \cap S_1(y_0)$ where $x,y \not = z$.
Thus,
\begin{align*}
F(\rho) &= \sum_{x,y} \rho(x,y)(1 - d(x,y)) \\
&= \sum_{x\neq x_0, y\neq y_0} \rho(x,y) (1-d(x,y)) + \rho(x_0,x_0) - \sum_{y \in B_{y_0}^{x_0}} \rho(x_0,y) + \rho(y_0,y_0)  - \sum_{x \in B_{x_0}^{y_0}} \rho(x,y_0)  \\
&=\# B_{x_0y_0} -   \sum_{x \in B_{x_0}^{y_0}} \rho(x,y_0) - \sum_{y \in B_{y_0}^{x_0}} \rho(x_0,y)
+ \sum_{x\in B_{x_0}^{y_0}, y \in B_{y_0}^{x_0}} \rho(x,y)(1-d(x,y)).   \end{align*}

If $x \in B_{x_0}^{y_0}$, then $x \sim x_0$ so that $\sum_y \rho(x, y)=1$ by \eqref{eq:TransportCombinatorial}.  Therefore, as $\rho(x,z)=0$ for all $z \in S_1(x_0) \cap S_1(y_0)$, either $\rho(x,y_0)=1$ or there exists a unique $y \in B_{y_0}^{x_0}$ such that $\rho(x,y)=1$.  In the second case, $y \in B_{y_0}^{x_0}$ is unique as $\sum_x \rho(x,y)=1$ by \eqref{eq:TransportCombinatorial}. Hence, $\rho$ can be uniquely associated with a bijection $\phi_\rho \in \Phi_{x_0y_0}$ by letting 
\[D(\phi_\rho) = \{ x \in B_{x_0}^{y_0} : \mbox{ there exists a unique } y \in B_{y_0}^{x_0} \mbox{ such that } \rho(x,y)=1 \} \] 
and $\phi_\rho(x) = y$ for $x \in D(\phi)$.

Note, by the dichotomy above, that $D(\phi_\rho)^c = \{ x \in B_{x_0}^{y_0} : \rho(x,y_0) =1 \}$ and $R(\phi_\rho)^c=\{ y \in B_{y_0}^{x_0} : \rho(x_0,y)=1\}$.  Therefore,
\begin{align*}
 \kappa(x_0,y_0) = F(\rho) &= \# B_{x_0y_0} - \left( \#D(\phi_\rho)^c + \#R(\phi_\rho)^c + \sum_{x \in D(\phi_\rho)} [d(x, \phi_\rho(x))-1] \right) \\
 &\leq \# B_{x_0y_0} - \inf_{\phi \in \Phi_{x_0y_0}} \left( \#D(\phi)^c + \# R(\phi)^c + \sum_{x\in D(\phi)}   [d(x,\phi(x)) - 1]  \right).
\end{align*}
On the other hand,  if $\phi \in \Phi_{x_0y_0}$, we can reverse the process above to define $\rho_\phi:B_1(x_0) \times B_1(y_0) \to \{0,1\}$ by letting $\rho_\phi(z,z)=1$ for all $z \in B_{x_0y_0}$, $\rho_\phi(x, \phi(x))=1$ for all $x \in D(\phi)$, $\rho_\phi(x,y_0)=1$ for all $x \in D(\phi)^c$, $\rho(x_0,y)=1$ for all $y \in R(\phi)^c$ and $\rho_\phi(x,y)=0$ otherwise.  As above, it follows that $\rho_\phi$ satisfies \eqref{eq:TransportCombinatorial} and that  $F(\rho_\phi) =  \# B_{x_0y_0} - \left( \#D(\phi)^c + \#R(\phi)^c + \sum_{x \in D(\phi)} [d(x, \phi(x))-1] \right).$
Therefore,
\begin{align*}
 \kappa(x_0,y_0) = \sup_\rho F(\rho) \geq F(\rho_\phi) &= \# B_{x_0y_0} - \left( \#D(\phi)^c + \#R(\phi)^c + \sum_{x \in D(\phi)} [d(x, \phi(x))-1] \right) 
 \end{align*}
 for all $\phi \in \Phi_{x_0y_0}$.
 Combining the two inequalities completes the proof.
\end{proof}

\subsection{Ollivier curvature on birth-death chains}\label{sec:OllivierBirthDeath}

The curvature of birth-death chains is easy to compute. Moreover, as we will see later, many problems of interest concerning Ollivier curvature can be reduced to the case of birth-death chains.

\begin{defn}
A graph $G=(\IN_0,w,m)$ is called a \emph{birth-death chain} if
$$
w(m,n) =0 \quad \mbox{ whenever } \quad |m-n|\neq 1.
$$
\end{defn}
\begin{theorem}[Curvature of a birth-death chain]\label{thm:line}
	Let $G=(\IN_0,w,m)$ be a birth-death chain and let $f(r):=d(0,r)=r$. Then for $0\leq r<R$,
\begin{align*}
	\kappa(r,R) &= \nabla_{rR}\Delta f = \frac{\Delta f(r) - \Delta f(R)}{R-r}  \\&= \frac{w(r,r+1) - w(r,r-1)}{(R-r)m(r)}   -  \frac{w(R,R+1) - w(R,R-1)}{(R-r)m(R)}
\end{align*}
where we set $w(r,r-1) :=0$ if $r=0$.	
\end{theorem}

\begin{proof}
	The last equality is a straightforward computation. We now prove the first equality.
	Due to Theorem~\ref{thm:nablaDelta}, as $f \in Lip(1)$ and $\nabla_{Rr}f=1$, it is clear that
	\begin{align*}
	\kappa(r,R) \leq \nabla_{rR} \Delta f = \frac{\Delta f(r) - \Delta f(R)}{R-r}.
	\end{align*}

We will now show the other inequality to complete the proof.	
Let $g \in Lip(1)$ be such that $\nabla_{Rr}g=1$, i.e., $g(R)-g(r)=R-r$.  Therefore, $g(n+1)-g(n)=1$ for all $r \leq n \leq R-1$ so that, in particular,
	$g(r+1)-g(r) = 1 = g(R)-g(R-1)$.
	Moreover, $a:=g(r)-g(r-1) \leq 1$ and $b:=g(R+1)-g(R) \leq 1$ since $g \in Lip(1)$.
	As
	$$
	m(r)\Delta g(r) = w(r,r+1) - a w(r,r-1) \geq m(r)\Delta f(r) 
	$$
	and
	$$
	m(R)\Delta g(R) = bw(R,R+1) -  w(R,R-1) \leq m(R)\Delta f(R)
	$$	
it follows that 
\begin{align*}
\frac{\Delta g(r) - \Delta g(R)}{R-r} \geq \frac{\Delta f(r) - \Delta f(R)}{R-r}.
\end{align*}
Therefore, Theorem~\ref{thm:nablaDelta} yields that
$$
	\kappa(r,R) \geq \frac{\Delta f(r) - \Delta f(R)}{R-r}
$$
which implies the claim of the theorem.
\end{proof}

\begin{rem}\label{rem:line}
We note that it is easy to see from the above that
\[ \kappa(0,r) = \frac{1}{r} \sum_{n=0}^{r-1} \kappa(n,n+1).\] 
In particular, $\kappa(r-1,r)=K$ if and only if $\kappa(0,r)=K$ for all $r \geq 1$.
\end{rem}

\section{Gradient estimates}\label{sec:GradEstimates}
Our proof of the gradient estimate of the semigroup under a Ricci curvature bound deeply relies on the maximum principle which requires taking maxima over compact sets. For applying this technique to infinite, and hence, non-compact graphs, we employ a cutoff method.
However, standard cutoff techniques like taking Dirichlet boundary conditions on a finite subgraph do not work since the gradient of a function may leave the subgraph.
Also cutting off with a finitely supported function after taking the semigroup appears to be not successful since we do not have control over the semigroup before taking the cutoff.

The idea to overcome these difficulties is to deeply intertwine the semigroup with a finitely supported cutoff function.
 We call this the perpetual cutoff method which will result in a non-linear cutoff semigroup whose general properties we first develop below.
For general theory on non-linear semigroups, see e.g. \cite{barbu1976nonlinear,kato1967nonlinear, miyadera1992nonlinear}.
We will then apply this general theory to prove our main characterization which connects a lower Ricci curvature bound with a gradient decay of the semigroup.

\subsection{The perpetual cutoff method}

The intuition of the non-linear cutoff semigroup presented below is that it behaves exactly as the heat semigroup whenever the heat does not surpass the cutoff threshold.
The name perpetual cutoff method comes from the fact that the cutoff threshold is not only applied once, but perpetually for all times $t>0$.

\begin{defn}[Cutoff semigroup]
Let $\phi \in C_c(V)$ be a non-negative function and let $f \in [0,\phi] := \{g \in C_c(V) : 0 \leq g \leq \phi\}$. 

For $t\geq0$, we define
\[
Q_t^\phi f :=  P_t f \wedge \phi
\]
and the \emph{cutoff semigroup}
\[
P_t^\phi f := \inf_{t_1 + \ldots + t_n = t} Q_{t_1}^\phi \ldots Q_{t_n}^\phi f.
\]
\end{defn}

We note, by checking cases, that $Q_t^\phi Q_s^\phi \leq Q_{t+s}^\phi$ and, as $P_t$ is positivity preserving, the infimum exists.

Let $W \subset V$ be finite.
We will show that $P_t^\phi f$ is a generalization of the semigroup $e^{t \Delta_W}$ with $\Delta_W f := 1_W \Delta (1_W f)$ corresponding to the Dirichlet problem $\partial_t u = \Delta u$ on $W$ and $u=0$ on $V \setminus W$.  In particular, $P_t^\phi = e^{t \Delta_W}$ when we take $\phi = 1_W$ as the cutoff function.  Furthermore, $P_t^\phi$ solves the heat equation at all vertices $x$ where $P_t^\phi(x) < \phi(x)$.

We collect these and some other useful properties of $P_t^\phi$ in the following theorem.
We write 
\[\overline {\partial_t^\pm} G(t) := \limsup_{h \to 0^\pm} \frac{G(t+h) -G(t)} h\] and \[\underline {\partial_t^\pm} G(t) := \liminf_{h \to 0^\pm} \frac{G(t+h) -G(t)} h\] for a function $G$ depending on $t$.  

\begin{theorem}[Properties of the cutoff semigroup]\label{thm:Cutoff}
Let $G=(V,w,m)$ be a graph and let $\phi \in C_c(V)$ be non-negative.
The family $P_t^\phi : [0,\phi] \to [0,\phi]$ is a nonlinear contraction semigroup with respect to $\|\cdot\|_p$ for all $p \in [1,\infty]$ and $t\geq0$. In particular, for $f,g \in [0,\phi]$ and $s,t \geq0$, we have:
\begin{enumerate}[(i)]
\item $P_t^\phi P_s^\phi = P_{t+s}^\phi$,
\item $\| P_t^\phi f - P_t^\phi g \|_p  \leq \|f-g\|_p$,
\item $P_0^\phi f = f$,
\item
$
P_t^\phi f \geq P_t^\psi g \quad  \mbox{ whenever } \quad  \phi \geq \psi \geq f \geq g,
$
\item $e^{-t \Deg} f   \leq  P_t^\phi f \leq P_t f$,
\item $P_t^\phi f$ is Lipschitz in $t$,
\item
$
\overline{\partial_t^\pm} P_t^\phi f  \leq \Delta  P_t^\phi f,
$
\item
$
\partial_t P_t^\phi f (x) = \Delta  P_t^\phi f (x)  \quad \mbox{ whenever } \quad    P_t^\phi f (x) <  \phi(x),
$
\item
$
P_t^\phi f = e^{t\Delta_W} f  \quad \mbox{ whenever } \quad \phi = 1_W \mbox{ for } W \subset V \mbox{ finite}.
$

\end{enumerate}
\end{theorem}


\begin{proof}
By definition, $Q_t^\phi$ maps $[0,\phi]$ to $[0,\phi]$, and so does $P_t^\phi$.
We prove the semigroup property $(i)$ by observing that
\begin{align*}
P_t^\phi P_s^\phi f =  \inf_{t_1 + \ldots + t_n = t} Q_{t_1}^\phi \ldots Q_{t_n}^\phi  \inf_{s_1 + \ldots + s_m = s} Q_{s_1}^\phi \ldots Q_{s_m}^\phi f &= 
\inf_{\substack{t_1 + \ldots + t_n = t \\ s_1 + \ldots + s_m = s}} Q_{t_1}^\phi \ldots Q_{t_n}^\phi Q_{s_1}^\phi \ldots Q_{s_m}^\phi 
\\
&= \inf_{t_1 + \ldots + t_{n+m} = t+s} Q_{t_1}^\phi \ldots Q_{t_{n+m}}^\phi f\\
&=P_{t+s}^\phi f.
\end{align*}
where the second equality follows from the montone convergence of $Q_t^\phi$ and the 
third equality follows from $Q_t^\phi Q_s^\phi f \leq Q_{t+s}^\phi f$.

To prove the contraction property $(ii)$, observe that for $p \in [1,\infty]$, $P_t$ is contracting on $\ell^p(V,m)$ so that
\[
\| Q_t^\phi f - Q_t^\phi g \|_p \leq \| P_t f - P_t g \|_p \leq \|f-g\|_p
\]
implying that
\[
\| P_t^\phi f - P_t^\phi g \|_p  \leq \|f-g\|_p.
\]
It is clear that $P_0^\phi f = f$ since $f \leq \phi$. This proves $(iii)$.

To prove $(iv)$, observe that $Q_t^\phi f \geq Q_t^\psi g$ whenever $\phi \geq \psi \geq f \geq g$ as $P_t$ is positivity preserving. This property is immediately transmitted to $P_t^\phi$.

To prove $(v)$, i.e., the lower and upper estimate of $P_t^\phi f$, we use
\[
e^{-t \Deg} f   \leq  Q_t^\phi f \leq P_t f
\]
which implies that
\[
e^{-t \Deg} f   \leq  P_t^\phi f \leq P_t f.
\]
as desired.

By using the estimates directly above and applying Taylor's theorem to $P_t f$ at $t=0$, we can deduce the existence of a constant $C_\phi>0$ such that for all $f \in [0,\phi]$ and $t\geq0$,
\begin{align*}
-C_\phi t \leq 1_{\supp \phi} (e^{-t\Deg} - 1) f \leq P^\phi_{t} f -  f \leq 1_{\supp \phi} (P_t -1) f \leq C_\phi t
\end{align*}
implying that $P_t^\phi f$ is Lipschitz in $t$ by using the semigroup property $(i)$, thus proving $(vi)$.

Furthermore, due to Taylor's theorem again, there exists a constant $C_\phi'>0$, such that for all $t>0$ and all $f \in [0,\phi]$,
\[
\frac 1 t \left( P^\phi_{t} f -  f   \right) \leq  1_{\supp \phi}  \frac 1 t (P_t f -f) \leq \Delta f + C_\phi't
\]
since $1_{\supp \phi} \Delta f \leq \Delta f$ for $f \in [0,\phi]$.  This directly implies that 
\[ \overline{\partial_t^+} P_t^\phi f = \limsup_{\eps \to 0^+} \frac 1 \eps \left({P_{\eps}^\phi P_t^\phi f - P_t^\phi f} \right) \leq \Delta P_t^\phi f \] 
by using the semigroup property and the fact that $P_t^\phi f \in [0,\phi]$.
Similarly,
\[ \overline{\partial_t^-} P_t^\phi f = \limsup_{\eps \to 0^-} \frac{1}{-\eps} \left( {P_{-\eps}^\phi P_{t+\eps}^\phi f - P_{t+\eps}^\phi f} \right) \leq \limsup_{\eps \to0^-} \left(\Delta P_{t+\eps}^\phi f - C_\phi' \eps \right) = \Delta P_{t}^\phi f. \] 
\eat{
Therefore,
\begin{align}
\frac1 \eps \left(P_t^\phi f - P_{t - \eps}^\phi f \right) \leq 1_{\supp \phi} \Delta P_{t-\eps}^\phi f + C_\phi'\eps \stackrel{\eps \to 0}{\longrightarrow} \Delta P_t^\phi f
\end{align}
This implies for $t>0$ and $f \in [0,\phi]$,
\[
\overline{\partial_t^-} P_t^\phi f \leq  \Delta P_t^\phi f.
\]
}
Putting these two inequalities together yields $(vii)$.

In order to do prove $(viii)$, we first define $\Delta^x: C(V)\to C(V)$ via
\[
\Delta^x f (y) := \begin{cases} \Delta f(x) & \mbox{if } y=x \\
-\Deg(y) f(y) & \mbox{otherwise}
\end{cases}
\]
and let $P_t^x := e^{t\Delta^x}$.
We remark that $\Delta^x$ is an asymmetrization of $\Delta$ and that $P_t^x u$  does \emph{not} give a solution to the Dirichlet problem $\partial_t u(x)=\Delta u(x)$ and $u=0$ on $V\setminus \{x\}$.  We also note that $P_t^x$ is positivity preserving.

\begin{lemma} \label{lem:Ptx}
Let $t>0$ and let $f \in [0,\phi]$.
If $P_s f(x) \leq \phi(x)$ for $0\leq s \leq t$, then
\[P_t^x f \leq P_t^\phi f. \]
\end{lemma}

\begin{proof}
Obviously, $P_s^x f \leq P_s f$.
Observe that $P_s^x f \leq \phi$ since
$P_s^x f(y) = e^{-s\Deg(y)}f(y) \leq f(y) \leq \phi(y)$  for $y \not = x$
and since $P_s^x f (x) = P_s f (x) \leq \phi(x)$ by assumption. Hence, $P_t^x f \leq Q_t^\phi f$.  Induction over $n$ for $s_1 + \ldots + s_n = t$ yields
\[
P_t^x f= P_{s_1 +\ldots + s_n}^xf = P_{s_1}^x \ldots P_{s_n}^x f\leq Q_{s_1}^\phi \ldots Q_{s_n}^\phi f \leq Q_t^\phi f. \]
Taking the infimum over all such $s_1, \ldots, s_n$ finishes the proof of the lemma.
\end{proof}

We now prove $(viii)$. Since we already proved $(vii)$,
it suffices to show that
\[
\underline{\partial_t^\pm} P_t^\phi f(x) \geq \Delta P_t^\phi f(x).
\]
whenever $P_t^\phi f(x) < \phi(x)$.

Due to Taylor's theorem with Lagrange remainder term, there exists a constant $C_\phi''>0$ such that for all $g \in [0,\phi]$ and all $\eps \in(0,t]$, there exists $\delta \in [0,\eps]$ such that
\begin{align}\label{eq:Taylor}
\frac{1}{\eps} (P_\eps^x g - g)(x) &= \Delta^x g(x) + \frac \eps 2  \partial_s^2 P_s^x g(x)|_{s=\delta} \nonumber
\\
&\geq \Delta g(x) - C_\phi'' \eps. 
\end{align}
since $\partial_s^2 P_s^x g(x) =  \Delta^x \Delta^x P^x_s g(x)$ is uniformly bounded on $[0,t]\times [0,\phi]$ and $x$ is fixed.

Choose $g=P_t^\phi f$.  Since we have assumed that $g(x)=P_t^\phi f(x) < \phi(x)$, by continuity of $P_s$, there exists $\eps \in (0,t]$ such that $P_sg(x) \leq \phi(x)$ for all $s \in [0,\eps]$.
By Lemma~\ref{lem:Ptx}, we then have that $P_\eps^x g \leq P_\eps^\phi g$
proving that $\underline{\partial_t^+} P_t^\phi f(x) \geq \Delta P_t^\phi f(x)$ by using \eqref{eq:Taylor}.

We next prove the same inequality for the left derivative.

We note that for $\eps<0$ small enough, we have that $P_{t+\eps}^\phi f(x)<\phi(x)$ so by using continuity of $P_s$ as above, we may apply Lemma~\ref{lem:Ptx} and \eqref{eq:Taylor} again to get
\begin{align*}
 \underline{\partial_t^-} P_t^\phi f(x) &\geq \liminf_{\eps \to 0^-} \frac{1}{-\eps} \left({P_{-\eps}^x P_{t+\eps}^\phi f(x) - P_{t+\eps}^\phi f(x)}\right) \\
 & \geq \liminf_{\eps \to 0^-} \left( \Delta P_{t +\eps}^\phi f(x) - C_\phi'' \eps \right) = \Delta P_{t}^\phi f(x).
 \end{align*}

\eat{
Choose $\eps$ small enough s.t. $P_s P_{t-\delta}^\phi f(x) < \phi(x)$ for $0< \delta,s < \eps$.
At vertex $x$, we calculate
\begin{align}
\frac{1}{\eps} \left( P_t^\phi f - P_{t-\eps}^\phi f \right) \geq \frac{1}{\eps} (P_\eps^x - 1) P_{t-\eps}^\phi f = \Delta P_{t-\eps}^\phi f - C_\phi''(\eps) \stackrel{\eps \to 0}{\longrightarrow} \Delta P_t^\phi f
\end{align}
where the estimate follows from Lemma~\ref{lem:Ptx}.
This proves $\underline{\partial_t^-} P_t^\phi f(x) \geq \Delta P_t^\phi f(x)$. }
Putting this together with $\underline{\partial_t^+} P_t^\phi f(x) \geq \Delta P_t^\phi f(x)$ and $(vii)$ yields $(viii)$.

We finally prove $(ix)$.
Let $\phi=1_W$.
First, we suppose that $f\leq \mathbf 1-\eps$. Then, $P_t^\phi f$ solves the Dirichlet problem $\partial_t u = \Delta u$ on $W$ and $u=0$ on $V \setminus W $ due to $(viii)$ as $P_sf(x) < 1 = 1_W(x)$ for $x \in W$ and $s \in [0,t]$ implies that $P_s^\phi f(x) < \phi(x)$ for all $x \in W$. 
This shows that $P_t^\phi f = e^{t\Delta_W} f$ since $e^{t\Delta_W} f$ is the unique solution to the Dirichlet problem.

For a general function $f \in [0,\phi]$, the claim follows by approximation since both $P_t^\phi$ and $e^{t\Delta_W}$
are contraction semigroups with respect to $\|\cdot\|_\infty$.
This proves $(ix)$ and finishes the proof of the theorem.
\end{proof}

\subsection{Cutoff semigroups and Ricci curvature}

Using the above observations, we can deduce a Lipschitz decay of the cutoff semigroup under lower curvature bounds.
We observe that the cutoff semigroup $P_t^\phi$ defined on $[0, \phi]=\{g \in C_c(V) : 0 \leq g \leq \phi \}$ canonically extends to functions $f: V \to [0,\infty)$ via $P_t^\phi f := P_t^\phi (f \wedge \phi)$.  In particular, $P_0^\phi f = f \wedge \phi$ whenever we do not assume that $f \leq \phi$.

\begin{lemma}\label{lem:CutoffRicci}
Let $G=(V,w,m)$ be a graph with $Ric(G) \geq K$.
Let $f:V \to [0,1]$ be non-constant, $T>0$ and $\phi:V \to [0,1]$ be compactly supported such that $\|\nabla \phi\|_\infty < \|\nabla f\|_\infty (1 \wedge e^{-KT})$. Then, for $t \in [0,T]$,
\[
\|\nabla P_t^\phi f\|_\infty \leq  e^{-Kt} \|\nabla f\|_\infty.  
\]
\end{lemma}

\begin{proof}
Without loss of generality, we can assume that $\kappa(x,y) > K$ for all $x,y \in V$ instead of $\kappa(x,y) \geq K$.  Furthermore, as we assume that $f$ is non-constant, it follows that $\|\nabla f\|_\infty>0$.

For $t \in [0,T], x,y \in V$ with $x \sim y$ we define
\begin{align*}
F(t,x,y) := e^{Kt} \nabla_{yx} P_t^\phi f.
\end{align*}
We aim to show that $F \leq \|\nabla f\|_\infty$.
Suppose not.

Since the support of $P_t^\phi f$ is contained in the finite support of $\phi$, the continuous function $F$ attains its maximum $F_{\max}$ at some $(t_0,x_0,y_0)$ where $y_0$ is in the support of $\phi$.  Therefore, $F(t_0,x_0,y_0) = F_{\max}  > \| \nabla f \|_\infty$.

Since \[F(0,x_0,y_0) = \nabla_{y_0x_0}(f \wedge \phi) \leq \|\nabla f\|_\infty \vee \|\nabla \phi\|_\infty = \|\nabla f\|_\infty < F(t_0,x_0,y_0),\]
we obtain that $t_0>0$.
Furthermore, observe that
\[
P_{t_0}^\phi f(x_0) < \phi(x_0)
\]
since otherwise 
\[
\nabla_{y_0x_0} P_{t_0}^\phi f \leq \nabla_{y_0x_0} \phi < \|\nabla f\|_\infty (1 \wedge e^{-KT}) \leq \|\nabla f\|_\infty e^{-Kt_0}
\]
which would imply that $F(t_0,x_0,y_0) < \|\nabla f\|_\infty$.

This yields $\partial_t P_{t}^\phi f(x_0)|_{t=t_0} = \Delta P_{t_0}^\phi f(x_0)$ due to Theorem~\ref{thm:Cutoff} $(viii)$. 
Moreover at $y_0$, Theorem~\ref{thm:Cutoff} $(vii)$ gives that $\overline{\partial_t^-} P_{t}^\phi f(y_0)|_{t=t_0} \leq  \Delta P_{t_0}^\phi f(y_0)$.
Subtracting yields
\begin{align*}
\underline{\partial_t^-} \nabla_{x_0y_0} P_t^\phi f |_{t=t_0} \geq \nabla_{x_0y_0} \Delta P_{t_0}^\phi f.
\end{align*}

Observe that $\|\nabla P_{t_0}^\phi f\|_\infty \leq F_{\max}e^{-Kt_0}$ and $\nabla_{y_0x_0}P_{t_0}^\phi f = F_{\max}e^{-Kt_0}$ due to maximality. Hence, due to Theorem~\ref{thm:nablaDelta}, we get that $F_{\max}e^{-Kt_0} \cdot \kappa(x_0,y_0) \leq \nabla_{x_0y_0} P_{t_0}^\phi f.$  Therefore, by our curvature assumption,
\begin{align*}
\underline{\partial_t^-}  \nabla_{x_0y_0} P_t^\phi f |_{t=t_0} \geq \nabla_{x_0y_0}\Delta P_{t_0}^\phi f \geq  F_{\max}e^{-Kt_0} \cdot \kappa(x_0,y_0) > F_{\max}e^{-Kt_0} K.
\end{align*}
Thus,
\begin{align*}
\overline{\partial_t^-}  F(t_0,x_0,y_0) &= \overline{\partial_t^-} \left(e^{Kt}\nabla_{y_0x_0} P_{t}^\phi f\right)|_{t=t_0} \\
&= KF_{\max} - e^{Kt_0} \underline{\partial_t^-} \nabla_{x_0y_0} P_t^\phi f |_{t=t_0} \\
&< KF_{\max} - KF_{\max} =0.
\end{align*}

Due to maximality in time of $F$ at $(t_0,x_0,y_0)$, since $t_0>0$, we have
$\overline{\partial_t^-}  F(t_0,x_0,y_0) \geq 0$ which contradicts the above inequality. 
Hence, $F\leq \|\nabla f\|_\infty$ which finishes the proof.
\end{proof}

\begin{lemma}\label{lem:RicImpliesGradient}
Let $G=(V,w,m)$ be a graph with $Ric(G) \geq K$.
Let $f : V \to [0,1]$ be non-constant. Then, for all $t>0$,
\[
\|\nabla P_t f\|_\infty \leq  e^{-Kt} \|\nabla f\|_\infty.  
\]
\end{lemma}

\begin{proof}
Let $T>0$. We prove the statement for all $t \in [0,T]$ which will prove the lemma.
Let $W_1 \subset W_2 \subset \ldots$ be finite subsets of $V$ such that $\bigcup W_n = V$.
Let $\phi_n :V \to [0,1]$ be functions such that $\phi_n = 1$ on $W_n$ and such that $\|\nabla \phi_n\|_\infty < \|\nabla f\|_\infty(1 \wedge e^{-KT})$.
Let $x\neq y \in V$ and $t \in [0,T]$. For all $n \in \IN$, Lemma~\ref{lem:CutoffRicci} yields
\[
\nabla_{xy} P_t^{\phi_n} f \leq  e^{-Kt} \|\nabla f\|_\infty.  
\]
Due to Theorem~\ref{thm:Cutoff} $(ix)$, we have $e^{t\Delta_{W_n}}= P_t^{1_{W_n}}$ on $[0,1_{W_n}]$, and since $1_{W_n} \leq \phi_n$, Theorem~\ref{thm:Cutoff} $(iv)$ yields
\[
e^{t\Delta_{W_n}}f = P_t^{1_{W_n}} f\leq P_t^{\phi_n} f \leq P_t f.
\]
Since $e^{t\Delta_{W_n}}f$ converges to $P_t f$ pointwise as $n \to \infty$, we infer that
\[
\nabla_{xy} P_t f = \lim_{n \to \infty} \nabla_{xy} P_t^{\phi_n} f \leq e^{-Kt} \|\nabla f\|_\infty.
\] 
Now the claim follows immediately since $x,y$ and $t$ are arbitrary.
\end{proof}

Using semigroup methods, we can now show that a lower curvature bound implies stochastic completeness. We want to point out that we will later independently prove stochastic completeness under even weaker assumptions using the Laplacian comparison principle (see Theorem~\ref{thm:StochComplete}).
\begin{lemma}\label{lem:StochComplete}
If $G=(V,w,m)$ is a graph with $Ric(G) \geq K$, then $G$ is stochastically complete.
\end{lemma}
\begin{rem}
We note that the proof closely follows the proof of stochastic completeness under a Bakry-Emery curvature bound in (\cite[Theorem~1.2]{hua2017stochastic}).
\end{rem}
\begin{proof}
Let $\eta_i:V \to [0,1]$ be non-constant such that $\eta_i \to 1$ pointwise  and $\|\nabla \eta_i\|_\infty \to 0$ as $i \to \infty$. Then for all $x\neq y$ and $t>0$, Lemma~\ref{lem:RicImpliesGradient} implies that
\begin{align*} 
\nabla_{xy} P_t \mathbf{1} = \lim_{i \to \infty} \nabla_{xy} P_t \eta_i \leq \lim_{i \to \infty} e^{-Kt}\|\nabla \eta_i\|_\infty =0.
\end{align*} 
Hence, $\|\nabla P_t  \mathbf{1}\|_\infty = 0$ which implies stochastic completeness as $P_0 \mathbf{1}=\mathbf{1}$.
\end{proof}

\subsection{Semigroup characterization}

Using Theorem~\ref{thm:nablaDelta} and Lemma~\ref{lem:RicImpliesGradient}, we now give a heat semigroup characterization of lower curvature bounds.
\begin{theorem}[Gradient of the semigroup]\label{thm:gradientGraphs}
	Let $G=(V,w,m)$ be a graph and let $K\in\IR$.  The following statements are equivalent:
	\begin{enumerate}[(1)]
		\item $Ric(G) \geq K$.
		\item For all $f \in C_c(V)$ and all $t>0$ $$\|\nabla P_t f\|_\infty \leq  e^{-Kt}\|\nabla f\|_\infty.$$ 
		\item  For all $f \in \ell_\infty(V)$ and all $t>0$ $$\| \nabla P_t f \|_\infty \leq e^{-Kt}  \| \nabla f \|_\infty.$$ 
		\item  $G$ is stochastically complete and for all $x,y \in V$ and all $t>0$
		\begin{align*}
		W(p^x_t,p^y_t) \leq e^{-Kt} d(x,y)
		\end{align*}
		where $p_t^x := \frac m {m(x)} P_t 1_x$ denotes the heat kernel.
	\end{enumerate}
	
\end{theorem}
\begin{rem}
We note that stochastic completeness is needed to state (4) since the Wasserstein distance $W$ is only defined on probability measures and $p_t^x$ is a probability measure only in the case of stochastic completeness.
\end{rem}

\begin{proof}

We first prove $(3) \Leftrightarrow (4)$.

For all bounded $1-$Lipschitz functions $f$, we have
\begin{align}
\int f dp^x_t -  \int f dp^y_t = \sum_{z \in V}f(z)\left(p_t^x(z)-p_t^y(z)\right) = P_t f(x) - P_t f(y) \label{eq:Wassertstein-Lipschitz}.
\end{align}
By definition, assertion $(4)$ is equivalent to
\begin{align*}
\int f dp^x_t -  \int f dp^y_t  \leq e^{-Kt} d(x,y)
\end{align*}
for all bounded Lipschitz functions which is equivalent to assertion $(3)$ due to
(\ref{eq:Wassertstein-Lipschitz}).
It is also clear that $(3)$ implies stochastic completeness by noting that $(3)$ implies that $\| \nabla P_t \mathbf{1} \|_\infty =0$.

The implication $(1)\Rightarrow (3)$ follows from Lemma~\ref{lem:RicImpliesGradient} if $f$ is non-constant and Lemma~\ref{lem:StochComplete} if $f$ is constant.

The implication $(3)\Rightarrow (2)$ is trivial.

We finally prove $(2) \Rightarrow (1)$.
Fix $x\sim y \in V$.
By Theorem~\ref{thm:nablaDelta},
it suffices to show that
$$\inf_{\substack {f\in Lip(1) \cap C_c(V)\\ \nabla_{yx}f=1}}  \nabla_{xy} \Delta f \geq K.$$


Let $f \in Lip(1)\cap C_c(V)$ be such that $\nabla_{yx}f =1$. 
By assertion $(2)$, we have 
$$\nabla_{yx} P_t f  \leq e^{-Kt}.$$

Hence, by taking the time derivative at $t=0$, 
\begin{align*}
\nabla_{xy} \Delta f = - \partial^+_t \nabla_{yx} P_t f |_{t=0} 
= \lim_{t\to 0^+} \frac 1 t \left( \nabla_{yx} f - \nabla_{yx} P_t f \right)
\geq \limsup_{t\to 0^+} \frac 1 t \left(1 - e^{-Kt} \right)
= K
\end{align*}
which proves assertion $(1)$ of the theorem since $f$ is arbitrary. 
\end{proof}

\section{Laplacian comparison principle}\label{sec:LaplaceCompare}

The classical Laplacian comparison theorem on manifolds compares the Laplacian of the distance function on the manifold to that of a model space with constant curvature. This means, for a given Riemannian manifold $M$ with Ricci curvature bounded from below by $K$ and for the model space $H$ with constant Ricci curvature $K$, one has
$$
\Delta^M d(x_0^M,\cdot) \leq \Delta^H d(x_0^H,\cdot).
$$
For a survey of comparison geometry of Ricci curvature on manifolds see \cite{zhu1997comparison}.

We give a discrete analogue of the above theorem in the sense that we upper bound the Laplacian of the distance function.
As a replacement of a model space, we will associate a birth-death chain to a given graph having the same sphere measure (see Section~\ref{sec:LaplaceCompareAndLineGrpahs}).  
We will also introduce a new quantity called the sphere curvature which depends only on the distance to a fixed vertex instead of considering all curvatures between neighbors.

We first give a discrete Laplacian comparison principle without a model space by explicitly estimating $\Delta d(x_0,\cdot)$.  
Even though the proof is a one-liner in light of Theorem~\ref{thm:nablaDelta}, the following discrete Laplacian comparison theorem, and its extension to the case of decaying curvature, turns out to be a foundation of a variety of applications, such as results concerning stochastic completeness and improved diameter bounds.

\begin{theorem}[Laplacian comparison]\label{thm:LaplaceCompare}
	Let $G=(V,w,m)$ be a graph. Let $x_0 \in V$ and suppose that $\kappa(x_0,\cdot) \geq K$ for some $K\in \IR$. Then,
	\begin{align*}
	\Delta d(x_0,\cdot) \leq \Deg(x_0) - Kd(x_0,\cdot).
	\end{align*}
	
	\begin{proof}
		Let $y \in V$, $y \not=x_0$, and set $f:=d(x_0,\cdot)$. Note that $f \in Lip(1)$ and $\nabla_{yx_0}f=1$ so that due to Theorem~\ref{thm:nablaDelta}, we have
		\begin{align*}
		K \leq \kappa(x_0,y) \leq \nabla_{x_0y} \Delta f = \frac{\Delta f(x_0)- \Delta f(y)}{d(x_0,y)} =  \frac{\Deg(x_0) - \Delta f(y)}{d(x_0,y)}.
		\end{align*}
		Rearranging yields the claim.
	\end{proof}
\end{theorem}

We next give a Laplacian comparison principle for decaying curvature.
To do so, we need to measure the minimal curvature in terms of the distance to some fixed vertex $x_0$.
\begin{defn}[Sphere curvatures]\label{def:sphereCurvature}
Let $x_0 \in V$ be a fixed vertex. By abuse of notation, we denote $S_{r} := S_{r}(x_0)$ and $B_r := B_r(x_0)$.
For $r\geq 1$, we let the \emph{sphere curvatures} be given by
\begin{align*}
\kappa(r) := \min_{y\in S_{r}} \max_{\substack{x \in S_{r-1} \\ x\sim y}} \kappa(x,y).
\end{align*}
\end{defn}
\begin{rem}
We remark that
$$
\kappa(r) \geq \min_{x,y \in B_r} \kappa(x,y)
$$
which describes the the curvature decay in a simpler way. However, for all of our results it will suffice to have a lower bound on $\kappa(r)$.
\end{rem}
\begin{theorem}[Laplacian comparison and decaying curvature]\label{thm:LaplaceComparisonNonConst}
   	Let $G=(V,w,m)$ be a graph, $x_0 \in V$ and $f:= d(x_0,\cdot)$. Then,
	\begin{align*}
	\Delta f \leq  \Phi(f)
	\end{align*}
	with 
	\begin{align*}
	\Phi(R) := \Deg(x_0) - \sum_{r=1}^{R} \kappa(r)
	\end{align*}
	for $R\geq1$ and $\Phi(0) = \Deg(x_0)$.
	The inequality is sharp for birth-death chains where we take $x_0=0$ so that $f(r)=d(0,r)=r$.
\end{theorem}
\begin{rem}
Note, in particular, that if $G$ is a graph with $Ric(G)\geq K$ and $H$ is a birth-death chain with $Ric(H)=K$ satisfying $\Deg_G(x_0)=\Deg_H(0)$, then 
\[ \Delta^G d(x_0, x) \leq \Delta^H d(0,R) \]
for all $x \in S_R$.  This makes the analogy to the statement concerning manifolds mentioned above precise.
\end{rem}

\begin{proof}
	We prove the result via induction over the radius $R$. The claim is clear for $R=0$ since $\Delta f(x_0) = \Deg(x_0)$.
	Let $R>0$ and let $y \in S_R$.  
	Let  $x \in S_{R-1}$ with $x\sim y$ be such that $\kappa(x,y)$ is maximal on $\{ (z,y) \ | \ z \in S_{R-1}, z \sim y \}$.
	Due to the definition of $\kappa(R)$ and Theorem~\ref{thm:nablaDelta}, we have
	\begin{align*}
	\kappa(R) \leq \kappa(x,y) \leq \nabla_{xy}\Delta f= \Delta f(x) - \Delta f(y).
	\end{align*}
	By the induction assumption, we have
	\begin{align*}
	\Delta f(x) \leq \Deg(x_0) - \sum_{r=1}^{R-1} \kappa(r).
	\end{align*}
	Rearranging and combining these yields
	\begin{align*}
	\Delta f(y) \leq \Delta f(x) -\kappa(R) \leq \Deg(x_0) - \sum_{r=1}^{R} \kappa(r)
	\end{align*}
	which proves the first statement.
	
	For birth-death chains, due to Theorem~\ref{thm:line}, we have for $r\geq 1$,
	\begin{align*}
	\kappa(r) = \kappa(r-1,r) = \Delta f(r-1) - \Delta f(r).
	\end{align*}
	Summing this up yields
	\begin{align*}
	\Delta f(R) =\Deg(0) - \sum_{r=1}^R \Delta f(r) = \Phi(R) = \Phi(f(R)).
	\end{align*}
	for all $R\geq 1$ which finishes the proof.
\end{proof}

\subsection{Curvature comparison and associated birth-death chains}\label{sec:LaplaceCompareAndLineGrpahs}

We now prove that the Laplacian comparison principle is compatible with the transition to birth-death chains.

\begin{defn}[Associated birth-death chain]\label{def:AssociatedlineGraph}
	Let $G=(V,w,m)$ be a graph with $x_0 \in V$ called the root vertex and let $S_r:=S_r(x_0)$.  We define the \emph{associated birth-death chain} $ \widetilde G = (\IN_0,\widetilde w, \widetilde m)$ via
	\begin{align*}
	\widetilde m(r) &:= m(S_r) \qquad \mbox{ and }\\ 
	\widetilde w(r,r+1)&:= w(S_r,S_{r+1}) := \sum_{\substack{x \in S_r \\y \in S_{r+1}}} w(x,y).
	\end{align*}
\end{defn}

\begin{theorem}[Associated birth-death chain and Laplacian comparison]\label{thm:AssociatedLaplaceComparisonLine}
	Let $G=(V,w,m)$ be a graph, $x_0 \in V$ and $f:=d(x_0,\cdot)$. Let $\widetilde G$ be the associated birth-death chain with Laplacian $\widetilde \Delta$ and $\widetilde f:=d(0,\cdot)$. Let $\Phi :\IN_0 \to \IR$ be a function. Then,
	\begin{align*}
	\Delta f \leq \Phi(f) \qquad \mbox{ implies } \qquad 	\widetilde \Delta \widetilde f \leq \Phi(\widetilde f).
	\end{align*}
\end{theorem}

\begin{proof}
We first note that $\Delta f(x_0) = \Deg(x_0)=\widetilde{\Deg}(0) = \widetilde{\Delta}\widetilde{f}(0)$.

Next, we let $r \in \IN$ and integrate $\Delta f \leq \Phi(f)$ over the sphere $S_r:=S_r(x_0)$.  For $x \in S_r$, we note that $m(x) \Delta f(x)= \sum_{y \in S_{r+1}} w(x,y) - \sum_{y \in S_{r-1}} w(x,y) \leq m(x) \Phi(r)$ so that
	\begin{align*}
	\Phi(r)\widetilde m(r) = \Phi(r) m(S_r) &= \sum_{x \in S_r} \Phi(r) m(x) \\&\geq \sum_{x \in S_r} \left( \sum_{y \in S_{r+1}} w(x,y) - \sum_{y \in S_{r-1}} w(x,y)  \right) \\
	&=\widetilde{w}(r,r+1) - \widetilde w(r,r-1).
	\end{align*}
	Hence,
	\begin{align*}
	\widetilde \Delta \widetilde f(r) = \frac{\widetilde{w}(r,r+1) - \widetilde w(r,r-1)}{\widetilde{m}(r)} \leq \Phi(r) = \Phi(\widetilde f(r))
	\end{align*}
	which finishes the proof.
\end{proof}

\eat{
We now prove that the Laplacian comparison is sharp on birth-death chains.
\begin{theorem}[Sharp Laplacian comparison and birth-death chains]\label{thm:LaplaceComparisonLine}
	Let $G=(\IN_0,w,m)$ be a birth-death chain and let $x_0=0 \in \IN_0$ be the root vertex. Let $f := d(x_0,\cdot)$. Then,
	\begin{align*}
		\Delta f =  \Phi(f)
	\end{align*}
	with 
	\begin{align*}
		\Phi(R) := \Deg(x_0) - \sum_{r=1}^{R} \kappa(r).
	\end{align*}
\end{theorem}

\begin{proof}
	Due to Theorem~\ref{thm:line}, we have for $r\geq 1$,
	\begin{align*}
	\kappa(r) = \kappa(r,r-1) = \Delta f(r-1) - \Delta f(r).
	\end{align*}
	Summing this up yields
	\begin{align*}
	\Delta f(R) =\Delta f(0) - \sum_{r=1}^R \Delta f(r) = \Phi(R) = \Phi(f(R)).
	\end{align*}
	for all $R\geq 0$ which finishes the proof.
\end{proof}
}

Combining this with the sharp Laplacian comparison for birth-death chains allows us to compare the curvature between a graph and its associated birth-death chain.
\begin{corollary}[Associated birth-death chain and curvature comparison]\label{cor:CurvCompare}
	Let $G=(V,w,m)$ be a graph, $x_0 \in V$ be a root vertex and $\kappa(r)$ be the sphere curvatures with respect to $x_0$. Let $\widetilde G= (\IN_0, \widetilde w, \widetilde m)$ be the associated birth-death chain with root vertex $\widetilde x_0 = 0$ and sphere curvatures $\widetilde \kappa(r) = \widetilde \kappa(r,r-1)$.
	Then,
	\begin{align*}
	\sum_{r=1}^R \widetilde \kappa(r) \geq 	\sum_{r=1}^R \kappa(r).
	\end{align*}
\end{corollary}

\begin{proof}
	Let $f:=d(x_0,\cdot)$ on $G$ and $\widetilde f := d(0,\cdot)$ on $\widetilde G$.
	Let 
	\begin{align*}
	\Phi(R) := \Deg(x_0) - \sum_{r=1}^{R} \kappa(r) \qquad \mbox{ and} \qquad
	\widetilde\Phi(R) := \widetilde \Deg(0) - \sum_{r=1}^{R} \widetilde \kappa(r)	.
	\end{align*}
	Due to Theorem~\ref{thm:LaplaceComparisonNonConst}, we have
	\begin{align*}
	\Delta f \leq \Phi(f) \qquad \mbox{ and } \qquad \widetilde \Delta \widetilde f = \widetilde \Phi(\widetilde f)
	\end{align*}
	Now, Theorems~\ref{thm:AssociatedLaplaceComparisonLine} yields
	\begin{align*}
	\widetilde \Phi(\widetilde f) = \widetilde \Delta \widetilde f \leq \Phi(\widetilde f)
	\end{align*}
	so that $\widetilde \Phi(R) \leq \Phi(R)$.  The fact that $\widetilde \Deg(0) = \Deg(x_0)$ completes the proof.
		
\end{proof}

One might be tempted to think that the sphere curvatures can also be compared without summation, i.e., $\widetilde \kappa(r) \geq \kappa(r)$ for all $r$. But this turns out to be wrong as demonstrated by the following example.
\begin{example}[Graph with $\widetilde\kappa(r) < \kappa(r)=0$]
Let $G=(\IZ,w,m)$ with root $x_0 = 0$ be given by
\begin{align*}
w(z,z+1) := m(z) := 2^z
\end{align*}	
and $w(m,n)=0$ if $|m-n| \neq 1$.
It is easy to see using the same techniques as in the proof of Theorem~\ref{thm:line} that $G$ has curvature $\kappa(r)= \kappa(r-1,r) = \Delta f(r-1) - \Delta f(r) = 0$ everywhere.

The associated birth-death chain $\widetilde{G} =(\IN_0, \widetilde w, \widetilde m)$ is then given by
\begin{align*}
\widetilde w(n,n+1) &= 2^n + 2^{-n-1} \mbox{ for } n\geq  0 \\
\widetilde m(n)&=2^n + 2^{-n} \quad \mbox{ for } n\geq  1  \qquad \mbox{ and } \qquad  \widetilde m(0)=1.
\end{align*}
Let $\widetilde f := d(0,\cdot)$ on $\widetilde G$. Thus, for $n\geq 1$,
$$
\widetilde \Delta \widetilde f(n) = \frac{\widetilde{w}(n,n+1) - \widetilde{w}(n,n-1)}{\widetilde{m}(n)} = \frac{2^{n-1} - 2^{-n-1}}{2^n + 2^{-n}}
$$
which is strictly increasing in $n$.
Hence for $r\geq 2$,
$$
\widetilde \kappa (r) = \widetilde \Delta \widetilde f (r-1) - \widetilde \Delta \widetilde f (r) < 0.
$$
\end{example}

\subsection{Stochastic completeness}

To prove stochastic completeness,
we will use the Khas'minskii criterion on graphs established by Huang in \cite[Theorem~3.3]{huang2011stochastic} which we restate now using our notation.

\begin{theorem}[Khas'minskii's criterion]\label{thm:Huang3.3}
Let $G=(V,w,m)$ be a graph. If there exists a non-negative function $f \in C(V)$
with
\begin{align*}
f(x) \to  \infty  \mbox{ as } \Deg(x) \to \infty
\end{align*}
satisfying
\begin{align*}
\Delta f  \leq \Psi(f)
\end{align*}
outside of a set of bounded vertex degree for some positive, increasing function $\Psi \in C^1([0,\infty))$ with 
$$
\int_0^\infty \frac{dr}{\Psi(r)} = \infty,
$$
then $G$ is stochastically complete.
\end{theorem}

Combining the Laplacian comparison with the Khas'minskii's criterion using $f = d(x_0, \cdot)$ yields an optimal stochastic completeness result.

\begin{theorem}[Stochastic completeness]\label{thm:StochComplete} \ \\
	\begin{enumerate}[(i)]
		\item If $G=(V,w,m)$ is a graph with 
		$$\kappa(r) \geq -C \log r$$ for some constant $C>0$ and large $r$, then
	 $G$ is stochastically complete. 
	\item  For $\eps>0$, let $G_\eps=(\IN_0, w,m)$ be a birth-death chain with $m\equiv1$ and 
	$$w(R,R+1) = 1 + \sum_{r=1}^R\sum_{k=1}^r  \left(\log k\right)^{1+\eps}.$$
	Then $G_\eps$ is stochastically incomplete and satisfies 
	$$\kappa(r) \geq - (\log r)^{1+\eps}$$ 
	for all $r\geq 2$. 
	\end{enumerate}
\end{theorem}

\begin{rem}
We note that the second statement shows that the first statement is optimal in the sense  that the decay rate $-\log r$ cannot be replaced by the faster decay rate $-( \log r )^{1+\eps}$.
\end{rem}

\begin{proof}
For the proof of $(i)$, let $f := d(x_0,\cdot)$.
Using the Laplacian comparison, Theorem~\ref{thm:LaplaceComparisonNonConst}, we have
$$
\Delta f  \leq \Phi(f)
$$
with 
$$\Phi(R) =  \Deg(x_0) -  \sum_{r=1}^{R} \kappa(r)  \leq \Psi(R)\in O(R \log (R))$$
since $-\kappa(r) \in O(\log(R))$, where $\Psi \in C^1([0,\infty))$ is some positive increasing function to which we can apply the Khas'minskii's criterion.
In particular,
$$
\int_0^\infty  \frac {dr}{\Psi(r)} = \infty,
$$	
so that Theorem~\ref{thm:Huang3.3} yields stochastic completeness as desired.

To prove $(ii)$, we let $f :=d(0,\cdot)$.
We first observe that for $R\geq1$,
\begin{align*}
\Delta f(R) = w(R,R+1) - w(R,R-1) = \sum_{k=1}^{R} (\log k)^{1+\eps}.
\end{align*}
Since $G_\eps$ is a birth-death chain, Theorem~\ref{thm:line} yields 
$$\kappa(R) = \kappa(R-1,R) = \Delta f(R-1) - \Delta f(R) = -(\log R)^{1+\eps}$$
for $R \geq 2$ as desired.

Since $\iint  (\log x)^{1+\eps} \in \Theta(x^2 (\log x)^{1+\eps})$, by definition of $w$, we have
\begin{align}
w(R,R+1) \in \Theta(R^2 (\log R)^{1+\eps}).  \label{eq:wThetaR}
\end{align}

Observe that as $G$ is a birth-death chain, it is weakly spherically symmetric with respect to $x_0 = 0 \in V= \IN_0$ in the sense of \cite[Definition~2.3]{keller2013volume}.
Hence, due to \cite[Theorem~5]{keller2013volume}, we know that $G$ is stochastically complete if and only if
$$
\sum_r \frac{r+1}{w(r,r+1)} = \infty.
$$
Due to (\ref{eq:wThetaR}), we have
$$
\frac {r+1}{w(r,r+1)} \in \Theta\left( \frac 1 {r (\log r)^{1+\eps}}\right)
$$
and since
$$\sum_r  \frac 1 {r (\log r)^{1+\eps}} < \infty$$
we have
$$
\sum_r \frac{r+1}{w(r,r+1)} < \infty
$$
which implies stochastic incompleteness.
\end{proof}

As mentioned in the introduction, the optimal curvature decay rate on Riemannian manifolds is of the order $-r^2$. 
As the use of intrinsic metrics has resolved various discrepancies between the manifold and graph settings in the past, 
one might think that using an intrinsic metric $\sigma$ instead of the combinatorial graph metric might give stochastic completeness when assuming $\kappa(r) \geq -C \sigma(0,r)^2$ in line with the manifolds case. This turns out to be wrong as we give an example of a stochastically incomplete graph with $\kappa(r)  \sim -(\log \sigma(0,r))^{1+\eps} $ for an intrinsic metric $\sigma$ where $f(n) \sim g(n)$ means $cf(n)<g(n)<Cf(n)$ for all $n \in \IN$ and some $C>c>0$.

We recall that a metric $\sigma$ on $V$ is called \emph{intrinsic} if 
\[ \Delta \sigma(x,\cdot)^2(x)=\frac 1 {m(x)}\sum_{y \in V} w(x,y)\sigma(x,y)^2 \leq 2 \]
for all $x \in V$.  For various uses the intrinsic metrics in the graph setting, see \cite{keller2015intrinsic}.

\begin{example}\label{ex:incompleteIntrinsic}
Let $G=(\IN_0,w,m)$  be a birth-death chain with 
$m(r)=2^r$  and  $w(r-1,r)=(\log r)^{1+\eps} \cdot r \cdot  2^r$ for $\eps>0$.
By Theorem~\ref{thm:line}, we obtain that $\kappa(r) \sim -(\log r)^{1+\eps}$.
Moreover, one can check that 
\[ \sigma(r,R) :=  \sum_{k=r}^{R-1} \Deg_+(k)^{-1/2} \]
gives an intrinsic metric where $\Deg_+(r) := w(r,r+1)/m(r) \sim r(\log r)^{1+\eps}$. In particular,
$\sigma(0,r) \sim \sqrt{r/(\log r)^{1+\eps}}$ and, thus,
$\kappa(r) \sim -(\log \sigma(0,r))^{1+\eps}$.

An objection to the example above is that the definition of the spherical curvature $\kappa$ depends on the combinatorial graph distance function $d$. 
However, in analogy to Theorem~\ref{thm:nablaDelta}, we can also define a curvature $\kappa_\sigma$ with respect to the intrinsic metric $\sigma$ via 
\[
\kappa_\sigma(x,y)= \inf\left\{\nabla_{xy}^\sigma \Delta f : \nabla_{yx}^\sigma f=1, \; \|\nabla^\sigma f\|_\infty = 1 \right\}
\]
where $\nabla_{xy}^\sigma f := \frac{f(x)-f(y)}{\sigma(x,y)}$.
On birth-death chains and intrinsic path metrics $\sigma$, for $x<y$ this simplifies to
\[
\kappa(x,y)= \nabla_{xy}^\sigma\Delta \sigma(0,\cdot).
\]
In our example, we have
\[
\Delta \sigma(0,\cdot)(r) \sim \sqrt{\Deg_+(r)} \sim \sqrt{r \cdot (\log r)^{1+\eps}}
\]
and, by using the mean value theorem to estimate the difference,
\begin{align*}
\kappa_\sigma(r,r+1) &=\sqrt{\Deg_+(r)} \cdot \left(\Delta \sigma(0,\cdot)(r)- \Delta \sigma(0,\cdot)(r+1) \right) \\
&\sim - \sqrt{\Deg_+(r)} \cdot \sqrt{\frac {(\log  r)^{1+\eps}}r} \\
&\sim - (\log r)^{1+\eps}.
\end{align*}
In particular, we also have $\kappa_\sigma(r) \sim -(\log \sigma(0,r))^{1+\eps}$. 
 
 We are left to show stochastic incompleteness.
 Due to \cite[Theorem~5]{keller2013volume}, $G$ is stochastically complete if and only if
\[
\sum_r \frac{m(\{1,...,r\})}{w(r,r+1)} = \infty.
\] 
However,
\[
\frac{m(\{1,...,r\})}{w(r,r+1)} \sim \frac 1{r (\log r)^{1+\eps}}
\]
which is summable. Therefore, $G$ is stochastically incomplete.
\end{example}

\subsection{Improved diameter bounds}

We prove that a graph with bounded degree and sphere curvatures decaying not faster than $1/R$ must be finite (Corollary~\ref{cor:finite}).  We also show that this decay rate is optimal (Theorem~\ref{thm:FiniteOptimal}).  For various diameter bounds on finite graphs see \cite{paeng2012volume}.

On the other hand, we show that in the case of unbounded degree, even a uniform positive lower curvature bound does not imply finiteness (see Example~\ref{Ex:positiveCurvInfiniteDiam}).
In contrast, if we assume that the measure is bounded from below, then a uniform positive lower curvature bound implies finiteness even in the case of unbounded degree (see Corollary~\ref{cor:FiniteDiamBoundedMeasure}).

As a warm-up, we start with the following diameter bound from \cite[Theorem~4.1]{lin2011ricci} transferred to our setting.
\begin{proposition}\label{prop:FiniteDiamBoundedDegree}
Let $G=(V,w,m)$ be a graph and let $x, y \in V$ with $x \not = y$.  If $\kappa(x,y)>0$, then
$$
d(x,y) \leq \frac{\Deg(x) + \Deg(y)} {\kappa(x,y)}.
$$
\end{proposition}
\begin{proof}
It is easy to see that $W(1_x,1_y) = d(x,y)$.  Furthermore,
observe that for sufficiently small $\eps$,
$$
W(1_x,m^\eps_x) = \eps \Deg(x).
$$
This follows as $W(1_x,m_x^\eps) = \sup_{f \in Lip(1)} - \eps \Delta f(x) \leq \eps \Deg(x)$ with equality for $f=1_x$.
Hence, by the triangle inequality,
\begin{align*}
W(m_x^\eps, m_y^\eps) &\geq W(1_x,1_y) - W(1_x,m^\eps_x) -W(1_y,m^\eps_y) \\
&= d(x,y) - \eps(\Deg(x) + \Deg(y)).
\end{align*}

Thus,
\begin{align*}
\kappa_\eps(x,y)   
&= 1 -  \frac{W(m_x^\eps, m_y^\eps)}{d(x,y)} \\
&\leq \eps \cdot  \frac{ \Deg(x) + \Deg(y)} {d(x,y)}.
\end{align*}
This yields the claim since $\kappa(x,y) = \lim_{\eps\to 0^+} \frac{1} {\eps}\kappa_\eps(x,y)$.
\end{proof}

In particular, if the degree is bounded and the curvature is uniformly positive, then the graph is finite.
More specifically, if we let $\diam(G)= \sup_{x,y \in V} d(x,y)$ denote the diameter of $G$, then if $\Deg(x)\leq M$ and $Ric(G)\geq K>0$, then 
$$\diam(G) \leq \frac{2M}{K}.$$
We now improve this result in the sense that we only lower bound the sphere curvatures, which allows for some negative curvature, and consider part of the vertex degrees.
For a fixed vertex $x_0 \in V$, we let for $x \in S_r:= S_r(x_0)$, 
\[ \Deg_{\pm}(x) = \frac{1}{m(x)} \sum_{y \in S_{r\pm1}} w(x,y) \]
denote the \emph{outer} and \emph{inner} degree of $x$.
Using the Laplacian comparison principle for non-constant curvature, we immediately obtain the following improved diameter bound.

 \begin{theorem}[Improved diameter bound]\label{thm:ImprovedDiamBound}
	Let $G=(V,w,m)$ be a graph with $x_0 \in V$. 
	If $S_R \not = \emptyset$ for $R>0$, then
	\begin{align*}
	\sum^R_{r=1} \kappa(r) \leq \Deg(x_0) + \min_{x \in S_R} \left(\Deg_-(x) - \Deg_+(x) \right).
	\end{align*}
In particular, if $\min_{x \in S_r} \left(\Deg_-(x) - \Deg_+(x) \right) \leq M$ and $\kappa(r) \geq K >0$ for all $r\geq1$, then
\[ \diam(G) \leq \frac{2(\Deg(x_0) +M)}{K}. \]
\end{theorem}
\begin{proof}
We recall that the Laplacian comparison, Theorem~\ref{thm:LaplaceComparisonNonConst}, gives that
\[ \Delta f(x) \leq \Deg(x_0) - \sum_{r=1}^R \kappa(r) \]
for $x \in S_R$ where $f(x) = d(x,x_0)$.  Now, the first statement follows as $\Delta f(x) = \Deg_+(x) - \Deg_-(x)$ by an easy calculation.
The second statement is an immediate consequence of the first statement and the triangle inequality.

\eat{
	We prove that there is no $x \in B_{R} = B_R(x_0)$ which will prove the result by the triangle inequality. Let $f:=d(x_0,\cdot)$.
	Suppose there exists an $x \in B_{R}$. Then, due to the Laplacian comparison principle for non-constant curvature (Theorem~\ref{thm:LaplaceComparisonNonConst}), we have
	\begin{align*}
	-\Deg(x) \leq \Delta f(x) \leq \Deg(x_0) - \sum_{r=1}^{R} \kappa(r) &< \Deg(x_0) -\left(\Deg(x_0) + \max_{y \in S_R} \Deg(y) \right) \\&\leq -\Deg(x).
	\end{align*}
	This is a contradiction which finishes the proof.}
\end{proof}

The theorem immediately gives us the following corollary.
\begin{corollary}\label{cor:finite}
	If $G=(V,w,m)$ is a graph with bounded degree, then 
	$$\limsup_{R \to \infty} \sum_r^R \kappa(r) < \infty.$$ 
Consequently, there is no infinite graph with bounded vertex degree satisfying 
$$\limsup_{R \to \infty} \sum_r^R \kappa(r)=\infty.$$
\end{corollary}

We show that the results above are optimal in the sense that whenever we have a given summable positive sequence $k_r$, we can find an infinite graph with bounded degree and 
summable sphere curvatures $\kappa(r)$ larger than $k_r$.

\begin{theorem}\label{thm:FiniteOptimal}
For every positive sequence $(k_r)_{r \in \IN}$ such that
$ \sum_r k_r < \infty$
there exists an infinite graph $G=(V,w,m)$ with bounded degree such that
$$ \kappa(r) \geq k_r \qquad \mbox{ and } \qquad \sum_r \kappa(r) <\infty.$$
\end{theorem}
\begin{proof}
We define a birth-death chain $G=(\IN_0,w,m)$ inductively with $w$ symmetric and $m$ satisfying $m(0)=1$, $w(0,1) = 2 \sum_{i>0} k_i$ 
and for $r\geq 1$,
	\begin{align*}
		m(r) = \frac{w(r,r-1)}{k_{r+1}}\qquad \mbox{ and}  \qquad w(r,r+1) = 2m(r) \sum_{i>r} k_i.  
	\end{align*}
Note, in particular, that $\frac{w(r,r-1)}{m(r)} = k_{r+1}$ while $\frac{w(r-1, r)}{m(r-1)}= 2 \sum_{i >r-1}k_i$.

	Due to Theorem~\ref{thm:line}, for $r>1$,
	\begin{align*}
	\kappa(r) &= \kappa(r-1,r) \\&= \frac{w(r-1,r)-w(r-1,r-2)}{m(r-1)}- \frac{w(r,r+1)-w(r,r-1)}{m(r)}			 \\
	&= 2 \sum_{i>r-1} k_i - k_r - 2 \sum_{i>r} k_i + k_{r+1} \\
	&=k_r + k_{r+1} \geq k_r
	\end{align*}
	which also shows that $\sum_r \kappa(r)<\infty$.
Similarly, $\kappa(1)= 2k_1 + k_2 \geq k_1$.

	It is left to show that the graph has bounded degree.
	We have
	\begin{align*}
	\Deg(r) = \frac{w(r,r-1)}{m(r)} + \frac{w(r,r+1)}{m(r)}  = k_{r+1} + 2 \sum_{i>r}k_i \leq 3 C
	\end{align*}
	with $C:= \sum_r k_r < \infty$	by assumption. This finishes the proof.
\end{proof}

\begin{example} \label{Ex:positiveCurvInfiniteDiam}
In contrast to Theorem~\ref{thm:ImprovedDiamBound}, we now show that there exist graphs with uniformly positive curvature which are infinite.  We note that all such graphs must have unbounded vertex degree. 

We construct an infinite birth-death chain $(\IN_0,w,m)$ such that $\kappa(x,y) = K>0.$
We first let $w(r,r+1)$ be strictly positive and decreasing in $r \in \IN_0$.  By Theorem~\ref{thm:line} and Remark~\ref{rem:line}, it suffices to find a choice of measure $m$ such that $\kappa(0,r)=K$, that is, for $f=d(0,\cdot)$ 
$$
 \Delta f(r) = \Delta f(0) - Kr = \Deg(0)-Kr.
$$

Choose $m(0)$ such that $\Deg(0) < Kr$ for all $r\geq 1$.  For this it suffices that $m(0) > \frac{w(0,1)}{K}$.
Then, for $r\geq 1$, choose
$$
m(r) := \frac{w(r,r-1) - w(r,r+1)}{Kr - \Deg(0)}
$$
guaranteeing
$$
\Delta f(r) = \frac 1 {m(r)} (w(r,r+1) - w(r,r-1)) = \Deg(0) - Kr.
$$
We remark that $m(r) >0$ since $w(r,r+1)$ is strictly decreasing. 

\end{example}

\subsection{Finiteness of the measure}

In this section, we show that a suited positive lower bound on the curvature implies finite measure, that is, $m(V):= \sum_{x\in V} m(x) < \infty$.

\begin{theorem}[Curvature and finite measure]
Let $G=(V,w,m)$ be a graph. If
\[
\liminf_{R\to \infty}\sum_{r=1}^R \kappa(r)>\Deg(x_0),
\]
then $m(V) < \infty$.
\end{theorem}

\begin{proof}
We first show that it suffices to prove the theorem for birth-death chains.
Let $\widetilde G = (\IN_0,\widetilde w, \widetilde m)$ be the birth-death chain associated to $G$.
Due to Corollary~\ref{cor:CurvCompare}, we also have 
\[
\liminf_{R\to \infty}\sum_{r=1}^R \widetilde\kappa(r)>\Deg(x_0)
\]
where $\widetilde \kappa(r)$ are the sphere curvatures of $\widetilde G$. Assuming that the theorem is proven for birth-death chains, we obtain that $m(V)=\widetilde m(\IN_0) < \infty$ which would finish the proof.

Now we prove the theorem for birth-death chains.
Let $f=d(0,\cdot)$. Due to Theorem~\ref{thm:LaplaceComparisonNonConst} and since $\liminf_{R\to \infty} \sum_r^R \kappa(r)> \Deg(0)$, we get
\[ \limsup_{R \to \infty} \Delta f(R) = \limsup_{R \to \infty} \left(\Deg(0) - \sum_{r=1}^R \kappa(r) \right) = \Deg(0) - \liminf_{R \to \infty} \sum_{r=1}^R \kappa(r) <0\]
so that
there exists $\eps>0$ and $R>0$ such that $\Delta f(r)\leq -\eps$ for all $r\geq R$.
This implies that
\[
\eps m(r) \leq w(r,r-1) - w(r,r+1)
\]
for $r\geq R$. Summing up, we obtain 
\[
 \eps \sum_{r=R}^\infty m(r) \leq w(R,R-1)
\]
which yields the finiteness of the measure of the birth-death chain.
This finishes the proof.
\end{proof}

The theorem immediately gives the following corollary.
\begin{corollary}\label{cor:FiniteDiamBoundedMeasure}
Let $G=(V,w,m)$ be a graph. If $\liminf_{R\to \infty}\sum_{r=1}^R \kappa(r) = \infty$, then $m(V)$ is finite. If, additionally, $\inf_{x \in V} m(x) >0$, then $G$ is finite.
\end{corollary}

Combining this with Corollary~\ref{cor:finite} we get the following dichotomy.
\begin{corollary}
Let $G=(V,w,m)$ be a graph and suppose that 
\[ \kappa(r)\geq \frac{C}{r} \]
for some $C>0$ and all large $r$.  Then either
$G$ is finite or $G$ is infinite with unbounded vertex degree and finite measure.
\end{corollary}


\section{Ricci curvature for continuous-time Markov processes}\label{sec:MarkovProcesses}

In this section, we compare our curvature notion to the curvature defined in \cite{veysseire2012coarse} for continuous time Markov processes, which generalize both locally finite graphs and Riemannian manifolds.
To make the comparison clear, we recall our curvature definition
\[
\kappa(x,y) =  \lim_{t \to 0^+} \frac 1 t \left(1 - \frac{W(m_x^t, m_y^t)}{d(x,y)} \right)
\]
where the discrete time
Markov kernel $m_x^t$ with laziness parameter $t \in(0,\infty)$ is given by
\[
\int{ f dm_x^t} = (f + t \Delta f)(x).
\]
Note that $m^t$ is only non-negative if the vertex degree is bounded and if $t$ is sufficiently small. By abuse of notation, we call $m^t$ a Markov kernel in any case.

The idea to define curvature in \cite{veysseire2012coarse} is to replace the measure $m_x^t$ by the continuous time heat kernel $p_x^t$ which has already appeared in Theorem~\ref{thm:gradientGraphs} and is
given by
\[
\int{ f dp_x^t} := P_t f (x).
\]
We note that this is equivalent to
\[
p_x^t (y)= P_t 1_y(x) = \frac{m(y)}{m(x)}P_t 1_x (y).
\]

Due to Taylor's theorem, it is reasonable to hope that $m_x^t$ is a good approximation for $p_x^t$ as $t \to 0^+$. Criteria for this approximation will be investigated in the next subsection.

Corresponding to  \cite[Definition~6]{veysseire2012coarse}, the coarse Ricci curvature on stochastically complete, continuous time Markov processes
is defined by
\[
\overline \kappa(x,y) := \limsup_{t \to 0^+} \frac 1 t \left( 1 - \frac{W(p_x^t, p_y^t)}{d(x,y)} \right)
\]
and
\[
\underline \kappa(x,y) :=  \liminf_{t \to 0^+} \frac 1 t \left( 1 - \frac{W(p_x^t, p_y^t)}{d(x,y)} \right).
\]
We recall that $\overline \kappa$ and $\underline \kappa$ do not coincide in general (see e.g. \cite[Example~8]{veysseire2012coarse}).
Furthermore, the above definition only makes sense in the stochastically complete case since, otherwise, $p_x^t$ is not a probability measure and, therefore, the Wasserstein distance is not well defined.

The main result of  \cite{veysseire2012coarse} is the equivalence of the lower curvature bound $\overline \kappa(x,y) \geq K$ and the Wasserstein contraction property
\[
W(p_x^t,p_y^t) \leq d(x,y)e^{-Kt}.
\]
We note that the same statement for a lower bound on $\kappa$ was shown in Theorem~\ref{thm:gradientGraphs}. 
We will show in Corollary~\ref{cor:MarkovChains} that $\kappa=\underline\kappa=\overline \kappa$ when assuming that $\kappa$ is uniformly bounded from below.  Therefore, the result in  \cite{veysseire2012coarse} combined with this equality gives an alternative proof of Theorem~\ref{thm:gradientGraphs}.

\subsection{Discrete and continuous time Markov kernels}
We will next give conditions guaranteeing that  discrete and continuous time Markov kernels approximate each other.

As a convenient notation, we extend the definition of the  semigroup $P_t$ to possibly unbounded non-negative functions.
\begin{defn}
For $f\geq 0$, we define \[P_t f :=  \sup_{\substack{0\leq g\leq f \\ g \in \ell_\infty(V)}} P_t g. \]
\end{defn}
The aim of this subsection is to prove that $W(p_x^t,m_x^t)=O(t^2)$ if and only if $P_t d(x,\cdot)<\infty$ for small $t>0$.
As a first step, we show a uniform boundedness property of the semigroup when applied to unbounded functions.
\begin{lemma}\label{lem:uniformFinitePt}
Let $x \in V$ and let $f \geq 0$. If $P_T f(x) < \infty$ for some $T>0$, then
\begin{enumerate}[(i)]
\item $
\sup_{t \in [0,T]} P_t f(x) < \infty.
$
\item $P_t f < \infty$ for all $t < T$.
\end{enumerate}
\end{lemma}

\begin{proof}
We first prove $(i)$.
Let $t \in [0,T]$ and let $g \in \ell_\infty(V)$ be such that $0\leq g \leq f$.
Then, \[P_t g(x) \leq e^{(T-t)\Deg(x)}P_T g(x) \leq e^{T\Deg(x)} P_T f(x) < \infty\] independently of $t$ and $g$.
Taking the supremum over $t \in [0,T]$ and $g$ yields $(i)$.

We now prove $(ii)$. Let $t<T$ and $y \in V$.
Let $g \in \ell_\infty(V)$ be such that $0\leq g \leq f$.
Observe that $P_T g(x) \geq  P_{T-t}1_y(x) \cdot  P_t g(y)$ where $P_{T-t}1_y(x)>0$ due to connectedness.
Hence, taking the supremum over $g$ yields
\[
P_t f(y) \leq \frac{P_T f(x)}{P_{T-t}1_y(x)} < \infty
\]
due to assumption. This proves $(ii)$ and finishes the proof of the lemma.
\end{proof}

We now characterize when the ball measure $m_x^t$ approximates the heat kernel measure $p_x^t$ as $t \to 0^+$.

\begin{proposition}\label{pro:PtdANDWpm} Suppose that $G=(V,w,m)$ is stochastically complete and
let $x \in V$.
The following statements are equivalent:
\begin{enumerate}[(1)]
\item $P_t d(x,\cdot) < \infty$ for some $t>0$.
\item $P_t f < \infty$ for all $f\geq 0$ with $\|\nabla f \|_\infty < \infty$ and some $t>0$.
\item $W(p_x^t,m_x^t) = O(t^2)$ as $t \to 0^+$.
\item $W(p_x^t,m_x^t) < \infty$ for some $t>0$.
\end{enumerate}
\end{proposition}
\begin{rem}
We remark that the above properties also play an important role as a standing assumption in \cite{joulin2007poisson} denoted by 
$P_t (x,\cdot) \in \mathscr{P}_1(E)$.
\end{rem}

\begin{proof}
The implication (1) $\Rightarrow$ (2) follows since $f \leq d(x,\cdot) + f(x)$ implies that $P_t f \leq f(x) + P_td(x,\cdot) < \infty$.

The implication (2) $\Rightarrow$ (1) is obvious.

We now show that (1) $\Rightarrow$ (3).
Due to Kantorovich duality, we have
\begin{align*}
W(p_x^t,m_x^t) &= \sup_{f \in Lip(1) \cap \ell_\infty(V)} \sum_{y \in V}     f(y) \left( p_x^t(y) - m_x^t(y) \right) \\
&= \sup_{f \in Lip(1) \cap \ell_\infty(V)} \sum_{y\in V}     f(y) \left( P_t 1_y(x) - 1_y(x) - t \Delta 1_y(x) \right)\\
&= \sup_{f \in Lip(1) \cap \ell_\infty(V)}  \left( P_t f - f - t \Delta f \right)(x).
\end{align*}
When optimizing, we can assume that $f(x)=1$ without loss of generality.   Since $f \in Lip(1)$, replacing $f$ by its positive part does not change the values of $f$ on $B_1(x)$ and does not decrease the values on $V\setminus B_1(x)$. 
Since   $p_x^t - m_x^t$ is non-negative on $V \setminus B_1(x)$, the objective function $\sum_y f(y) \left( p_x^t(y) - m_x^t(y) \right)$ is not decreased when replacing $f$ by its positive part. Therefore, we can assume that $f(x)=1$ and $f\geq 0$ when optimizing. This gives
\begin{align}\label{eq:WPepsmEps}
W(p_x^t,m_x^t) &  =  \sup_{\substack{f \in Lip(1) \cap \ell_\infty(V)  \\ f\geq 0, f(x)=1}}  (P_t f - f - t \Delta f)(x). 
\end{align}

Therefore, let $f \in Lip(1) \cap \ell_\infty(V)$ with $f\geq0$ and $f(x)=1$.
Then, $0 \leq f \leq 1 + d(x,\cdot).$
Due to Lemma~\ref{lem:uniformFinitePt} $(i)$ and by assumption, there exists $C>0$ such that $P_t(1 + d(x,\cdot)) \leq C$ on $B_2(x)$ for all $t \in [0,T]$.
Thus, we also have $0 \leq P_t f \leq C$ on $B_2(x)$ for all $t \in [0,T]$. This yields the existence of $C'$ independent of $f$ and $t \in [0,T]$ such that
\[
|\Delta \Delta P_t f| \leq C'.
\]

Due to Taylor's theorem, there exists  $\delta \in [0,t]$ such that
\begin{align*}
(P_t f - f  -  t \Delta f)(x) = \frac {t^2} 2 \Delta \Delta P_\delta f(x) \leq \frac {t^2} 2 C' = O(t^2).
\end{align*}
Putting this together with (\ref{eq:WPepsmEps}) proves that (1) $\Rightarrow$ (3).

The implication (3) $\Rightarrow$ (4) is obvious.

We now show that (4) $\Rightarrow$ (1).
Let $f = d(x,\cdot)+1$ and $f_n := f \wedge n \in \ell_\infty(V)$.
Due to (\ref{eq:WPepsmEps}), we have 
\[
\infty > W(p_x^t,m_x^t) \geq (P_t f_n - f_n -t \Delta f_n)(x)
\]
yielding 
\[
P_t f(x) = \sup_n P_t f_n(x) \leq W(p_x^t,m_x^t) + f(x) + t \Delta f(x) <\infty.
\]
Thus by Lemma~\ref{lem:uniformFinitePt} $(ii)$,
$P_s f < \infty$ for all $s < t$ as desired.
This finishes the proof.
\end{proof}

\subsection{Another Ricci curvature characterization}
We now prove that on locally finite graphs with Ricci curvature bounded from below, our definition of $\kappa$ coincides with $\overline{\kappa}$ and $\underline{\kappa}$ as defined in \cite[Definition~6]{veysseire2012coarse}.
This will yield another characterization of lower Ricci curvature bounds by combining with \cite[Theorem 9]{veysseire2012coarse}.

As a preparation, we show the subexponential behavior of non-negative $\lambda$-subharmonic functions under the heat equation.
\begin{lemma}\label{lem:subharmonicPt}
Let $f \geq 0$ satisfy $\Delta f \leq \lambda f$ for some $\lambda >0$. Then, $P_t f \leq e^{\lambda t} f$.
\end{lemma}

\begin{proof}
Let $W \subset V$ be finite.
Let $f_W:= f 1_W$ and let $e^{t\Delta_W}$ be the semigroup corresponding to $\Delta_W$ with $\Delta_W g := 1_W \Delta (g 1_W)$ representing Dirichlet boundary conditions.
First, we observe that $\Delta_W f_W \leq \lambda f_W$ since $\Delta f \leq \lambda f$ and $f \geq 0$.

Let $\phi := e^{t\left(\Delta_W-\lambda \right)} f_W$. Then, 
\[
\partial_t \phi = e^{-\lambda t} \left( \Delta_W -\lambda \right) e^{t\Delta_W} f_W = e^{-\lambda t} e^{t\Delta_W} \left( \Delta_W -\lambda \right) f_W \leq 0  
\]
showing that $\phi(t)=e^{-\lambda t}e^{t \Delta_W} f_W \leq  f_W = \phi(0)$.
Since $e^{t_W} f_W \to P_t f$ pointwise as $W\to V$, we obtain the desired claim that $P_t f \leq e^{\lambda t} f$.
\end{proof}
\begin{rem}
We remark that the step in the proof above where we take Dirichlet boundary conditions is necessary to ensure that $\Delta P_t f = P_t \Delta f$ which generally only holds true on the domain $D(\Delta) \subseteq \ell^2(V)$ on which $\Delta$ is self-adjoint.
\end{rem}

We next prove that a lower Ricci curvature bound implies that $P_t d(x,\cdot)<\infty$.
\begin{lemma}\label{lem:Ptd} 
Let $x \in V$ and $f := d(x,\cdot)$.  If $Ric(G) \geq -K$ for some $K >0$,
then 
\[P_t f \leq e^{Kt} (f + \Deg(x)/K)<\infty.\]
\end{lemma}
\begin{proof}
Due to the Laplacian comparison principle, Theorem~\ref{thm:LaplaceCompare}, we have that
\[
\Delta(f + \Deg(x)/K) = \Delta f \leq K(f + \Deg(x)/K).
\]
Thus, Lemma~\ref{lem:subharmonicPt} yields
\[
P_t f \leq P_t (f + \Deg(x)/K) \leq e^{Kt} (f + \Deg(x)/K)\]
as desired.
\end{proof}

We now present the main theorem of this section.

\begin{theorem}[Continuous and discrete time curvature]\label{thm:MarkovChains}
Let $G=(V,w,m)$ be a stochastically complete graph. Suppose that $P_t d(x_0,\cdot) < \infty$ for some $x_0 \in V$ and some $t>0$.  Then, for all $x \neq y$,
\[
\kappa(x,y) = \lim_{t \to 0^+} \frac 1 t \left(1 - \frac{W(p_x^t,p_y^t)}{d(x,y)} \right) = \overline\kappa(x,y) = \underline \kappa(x,y).
\]
\end{theorem}


\begin{proof}
Let $x \neq y \in V$.
Due to the triangle inequality and Proposition~\ref{pro:PtdANDWpm}, we have
\[
W(m_x^t,m_y^t) = W(p_x^t,p_y^t) + O(t^2)
\]
as $t\to 0^+$.
By definition,
\begin{align*}
\kappa(x,y) = \lim_{t\to 0^+} \frac 1 t \left(1 - \frac{W(m_x^t,m_y^t)}{d(x,y)} \right) &=\lim_{t\to 0^+} \frac 1 t \left(1 - \frac{W(p_x^t,p_y^t) + O(t^2)}{d(x,y)} \right)\\
&=\lim_{t \to 0^+} \frac 1 t \left(1 - \frac{W(p_x^t,p_y^t)}{d(x,y)} \right).
\end{align*}
This finishes the proof.
\end{proof}
Since both stochastic completeness and $P_t d(x_0,\cdot)<\infty$ are implied by a lower Ricci curvature bound (see Theorem~\ref{thm:StochComplete} and Lemma~\ref{lem:Ptd}), we immediately obtain the following corollary.

\begin{corollary}\label{cor:MarkovChains}
Let $G=(V,w,m)$ be a graph with $Ric(G) \geq K$ for some $K \in \IR$. Then, for all $x \neq y$,
\[
\kappa(x,y) = \lim_{t \to 0^+} \frac 1 t \left(1 - \frac{W(p_x^t,p_y^t)}{d(x,y)} \right).
\]
\end{corollary}
Combining with \cite[Theorem~9]{veysseire2012coarse}, we immediately obtain that $W(p_t^x,p_t^y) \leq e^{-Kt} d(x,y)$ whenever $Ric(G)\geq K$.
We remark that this gives an alternative method for proving Theorem~\ref{thm:gradientGraphs}.

\TOCstop

\subsection*{Acknowledgments}
F.M. wants to thank the German National Merit Foundation for financial support.
R.K.W. gratefully acknowledges financial support from PSC-CUNY Awards, jointly funded by the Professional Staff Congress and the City University of New York, and the Collaboration Grant for Mathematicians, funded by the Simons Foundation.
R.K.W. would also like to thank Fudan and Hokkaido Universities for their generous hospitality while parts of this work were completed.
Furthermore, both authors want to thank the Harvard University Center of Mathematical Sciences and Applications for their hospitality.

\printbibliography

Florentin M\"unch, \\
Department of Mathematics,
University of Potsdam, Potsdam, Germany\\
Currently: Center of Mathematical Sciences and Applications, Harvard University, Cambridge MA, USA \\
\texttt{chmuench@uni-potsdam.de}\\
\\
Rados{\l}aw K. Wojciechowski,\\
York College and the Graduate Center of the City University of New
York, New York, USA\\
\texttt{rwojciechowski@gc.cuny.edu}

\TOCstart

\end{document}